\documentclass{article}
\usepackage[utf8]{inputenc}
\usepackage{amsfonts}
\usepackage{amsmath}
\usepackage{amssymb}
\usepackage{mathtools}
\usepackage{amsthm}
\usepackage{dsfont}
\usepackage{enumerate}
\usepackage{color}
\usepackage[margin=1.35in]{geometry}
\usepackage[normalem]{ulem}
\usepackage{authblk}
\usepackage{mathscinet}
\allowdisplaybreaks 

\usepackage[
backend=biber,
style=alphabetic,doi=false,isbn=false,url=false,eprint=false
]{biblatex}
\usepackage{hyperref}
\usepackage{xurl}
\hypersetup{breaklinks=true}

\newcommand{\R}{\mathbb{R}}
\newcommand{\Q}{\mathbb{Q}}
\newcommand{\N}{\mathbb{N}}
\newcommand{\E}{\mathbb{E}}
\newcommand{\Prob}{\mathbb{P}}

\DeclareMathOperator{\Id}{Id}
\DeclareMathOperator{\Tr}{Tr}

\newtheorem{theorem}{Theorem}[section]

\newtheorem{definition}[theorem]{Definition}
\newtheorem{lemma}[theorem]{Lemma}
\newtheorem{assumption}[theorem]{Assumption}
\newtheorem{remark}[theorem]{Remark}

\title{A suitable nonlinear Stratonovich noise prevents blow-up in the Euler equations and other SPDEs}
\author{Marco Bagnara\footnote{Scuola Normale Superiore, Piazza dei Cavalieri 7, 56126 Pisa, Italy; \href{mailto:marco.bagnara@sns.it}{marco.bagnara@sns.it}}}

\date{}

\begin{document}

\maketitle

\begin{abstract}
We perturb the 3D Euler equations by a particular non-linear Stratonovich noise. We show the existence and uniqueness of a global-in-time (i.e.\ no blow-up) smooth solution. The result is a corollary of a more general theorem valid in an abstract framework, where the addition of such noise prevents the blow-up possibly induced by a drift with super-linear growth. The result is new with Stratonovich noise.
\end{abstract}



\tableofcontents

\section{Introduction}
\subsection{Deterministic and stochastic Euler equations}
Consider the incompressible Euler equations on a smooth, bounded and simply connected domain $D \subset \R^d$, with $d=2,3$. The Euler equations describe the motion of an incompressible, non-viscous fluid. The equations read as
\begin{align}\label{eq:det_Euler}
\begin{cases}
    \partial_t u + (u\cdot \nabla) u + \nabla p =0 &(t,x)\in [0,T]\times D,\\
    \operatorname{div} u=0 &(t,x)\in [0,T]\times D,\\
    u\cdot n=0 &(t,x)\in [0,T]\times \partial D,\\
    u(0,x)=u_0(x) & x \in D,
\end{cases}  
\end{align}
where $u:[0,T]\times D\to \R^d$ is the velocity field, $p:[0,T]\times D\to \R$ is the pressure field and $n$ is the outer unit normal on $\partial D$. If $\mathcal{P}$ denotes the so-called Leray projector, by applying $\mathcal{P}$ to \eqref{eq:det_Euler}, one obtains a Cauchy problem for the velocity field $u$
\begin{align}\label{eq:det_Euler_Cauchy}
\begin{cases}
    \partial_t u + \mathcal{P}(u\cdot \nabla) u =0 &(t,x)\in [0,T]\times D,\\
    u(0,x)=u_0(x) & x \in D,
\end{cases}    
\end{align}
and, given a solution $u$ of \eqref{eq:det_Euler_Cauchy}, can retrieve $p$ a posteriori.

Local well-posedness for the deterministic equations has been established in both Sobolev and  H\"older spaces (see \cite{Lichtenstein1925, gunther1927motion, EbiMar1970, Kato1988CommutatorEA}). However, despite in the two-dimensional case global well-posedness results have been achieved \cite{Wolibner1933, Yud1963}, in three dimensions it is still unknown whether smooth solutions exist globally. Actually, recent results suggest a negative answer to this questions (see, for instance, \cite{Elg2021, CheHou2022}).

In the stochastic case, global well-posedness for the stochastic Euler equations in two dimensions has been proven in several function spaces, see \cite{doi:10.1080, BesFla1999, BrzPes2001, 10.2307/2667278, doi:10.1137} 
and many others. On the other hand, in the three-dimensional setting, local well-posedness has been shown, for instance, by Kim \cite{Kim2009} under additive noise conditions, while Glatt-Holtz and Vicol \cite{GlaVic2014} showed local well-posedness for the stochastic Euler equations driven by a diverse range of multiplicative noise. Glatt-Holtz and Vicol also proved that the solution is global with high probability when the multiplicative noise is linear.

We refer to \cite{BMX2023} for a more in-depth outline on the existing literature.

\subsection{No blow-up by noise}

The lack of a global well-posedness result in the deterministic 3D Euler equations has lead to the question of whether noise can restore well-posedness. These kind of results are known in the literature as regularization by noise phenomena. Here, in particular, we explore whether a suitable noise can prevent solutions from blowing up. In the finite-dimensional framework, the question has been explored in, among others, \cite{doi101080Sch}, \cite{10.1214/EJP.v20-4047}, \cite{Athreya2012}, \cite{4484185}, \cite{gard1988introduction}, and, for Stratonovich noise, \cite{maurelli2020non}. 

In this paper, we address an infinite-dimensional blow-up problem associated with a strong superlinear noise. The approach is based on finding a suitable Lyapunov function, a technique inspired by finite-dimensional no-blow-up results. Similar works have been conducted in this direction, exploiting a superlinear It\^o noise. For example Alonso-Or\'an, Miao, and Tang in \cite{ALONSOORAN2022244} used this method to establish a no-blow-up result for a one-dimensional transport-type PDE with a nonlinear diffusion coefficient.  Ren, Tang, and Wang employed Lyapunov fucntions in \cite{RTW2020} for SPDEs characterized by a (possibly irregular and distribution-dependent) drift with superlinear growth and applied it to certain transport-type SPDEs. In \cite{TW2022}, Tang and Wang gave a general criterion applicable to a broad class of singular SPDEs, including the stochastic Euler and Navier-Stokes equations. In \cite{BMX2023}, Bagnara, Maurelli and Xu showed an example of no blow-up of Euler equations, as a by-product of a more general framework applicable to hyperbolic-type SPDEs.  Other approaches has also been explored, see for instance \cite{GlaVic2014}, \cite{BarRocZha2017}, \cite{flandoli2021high}. In general, we refer to the introduction in \cite{BMX2023} for a more detailed overview on the existing literature.

All the previously mentioned works which achieve a no blow-up result in infinite dimension by means of a strong super-linear noise fall in the category of It\^o noise. This is not just a stylistic choice in writing the articles, but more the fact that the technique relies on the It\^o second-order correction coming from It\^o lemma. Indeed, it is easy to convince themselves that the same noise, but in a Stratonovich integration framework, is not able to achieve the same results (at least with a similar proof). The point is that, in the Stratonovich case, the It\^o-Stratonovich correction and the It\^o second-order correction coming from those noises compete with each other preventig to obtain the key estimate in the Lyapunov function method.

In this note, we provide the first example (to our knowledge) of super-linear Stratonovich noise which is able to avoid blow-ups in infinite dimension in presence of a super-linear drift. A reason behind the relevance of the Stratonovich integration is the Wong-Zakai approximation theorem, which shows that the solution to the Stratonovich SDE is the limit of solutions to random ODEs associated to a regularized approximation of the noise \cite[Chapter 5.2, Proposition 2.24]{Karatzas-Shreve}. 

The shape of the noise is inspired by the work \cite{maurelli2020non} in the finite-dimensional context. Additionally, we note a conceptual similarity to the framework of stabilization by noise developed in \cite{Arnold-Crauel-Wihstutz} for linear systems (see the following subsection).

\subsection{Informal statements}

Let $E_1$, $E_0$, and $E_{-1}$ be three separable Hilbert spaces such that $E_1$ is compactly embedded in $E_0$ and $E_0$ is continuously embedded in $E_{-1}$. Let $(\Omega,\mathcal{A},(\mathcal{F}_t)_t,\Prob)$ be a filtered probability space. On such probability space, let $W$ be a $E_0$-valued $Q$-Wiener process, for a trace-class operator $Q:E_0 \to E_0$.
Consider the SDE,
\begin{align}
dX_t = b(X_t)dt +\sigma(X_t) \circ dW_t,\label{eq:SDE_intro}
\end{align}
where $b: E_0 \to E_{-1}$ is a Borel function, with $b(E_1)\subset E_0$, $\sigma: E_0 \to L(E_0,E_0)$ is a $C^1$ function and $\circ$ denotes the Stratonovich stochastic integration.
Given $R>0$, we require the diffusion coefficient to take the following form
\begin{align} \label{eq:sigma_intro}
    \sigma(x) = c_N\lVert x \rVert^\eta_{E_0} \left( I - \alpha  \frac {x \otimes x}{\lVert x \rVert^2_{E_0}} \right) \qquad \forall x \in B_{R;E_0}^c, 
\end{align}
for some real positive parameters $c_N$, $\eta$ and $\alpha$.

\begin{theorem}[informal, see Theorem \ref{thm:main}] \label{thm:main_intro}
Under some technical conditions for existence and uniqueness, assume that there exist $c_D>0$ and $m>2$ such that $\forall x \in B_{R;E_0}^c \cap E_1$ 
\begin{align}\label{eq:drift_bound_intro}
    \langle b(x),x \rangle_{E_0} \le c_D\lVert x \rVert^{m}_{E_0}.
\end{align}
Moreover, in \eqref{eq:sigma_intro}, let $\eta> m/2$, $c_N>0$ (or alternatively $\eta =m/2$, $c_N$ large enough) and $\alpha$ such that
\begin{align}\label{eq:alpha_condition_intro}
    1 < \alpha < 1+ \frac{1}{\eta-1} \frac{\Tr Q - \lVert Q \rVert}{\lVert Q \rVert}.
\end{align} 
Then, for every initial condition $x_0 \in E_1$, there exists a unique, global-in-time, strong solution of \eqref{eq:SDE_intro} in $L^\infty([0,T];E_1)\cap C([0,T];E_0)$, $\Prob$-a.s., for all $T>0$.
\end{theorem}

Being $m>2$, condition \eqref{eq:drift_bound_intro} encodes the super-linearity of the drift, that can be responsible of blow-ups in the deterministic setting. However, under some mild assumptions on existence and uniqueness, Theorem \ref{thm:main_intro} ensures that there always exists a Stratonovich noise which prevents the blow-up of the infinite dimensional SDE. In particular, the diffusion coefficient is super-linear itself, i.e. $\eta\ge m/2>1$. Interestingly, this result highlights a fundamental difference with respect to the It\^o cases known in the literature. Indeed, condition \eqref{eq:alpha_condition_intro} to be satisfied requires that $\Tr Q > \lVert Q \rVert$, where $\lVert Q \rVert$ denotes the operatorial norm of $Q$ as map from $E_0$ into $E_0$. This means that the $Q$-Wiener process $W$ must be at least two-dimensional, meaning that $Q$ must have at least two non-zero eigenvalues (counted with multiplicity), but possibly with one much smaller than the other. On the other hand, similar results obtained exploiting a super-linear It\^o noise can be achieved even with a $1$-dimensional Wiener process.

Despite the main interest of this work is the infinite dimensional setting, Therorem \ref{thm:main_intro} can be also applied to the finite dimensional setting providing a generalization of the result in \cite{maurelli2020non}. In particular, Therorem \ref{thm:main_intro} shows that in $\R^n$ blow-ups can be prevented also with Brownian motions of dimension smaller than $n$ (but two at minimum).

Analysing the shape of the diffusion coefficient in \eqref{eq:sigma_intro}, we note a different amplitude in the radial direction with respect to all other orthogonal directions (rotational directions). Here, we observe a similarity with the framework of stabilization by noise for a linear system found in \cite{Arnold-Crauel-Wihstutz}. Arnold, Crauel and Wihstutz consider diffusion coefficients of rotation type in a Stratonovich framework and are able to prove stabilization when the drift matrix has negative trace. In our work, beside the Stratonovich noise, we also have as many rotation components in the diffusion matrix as the dimensionality of the Wiener process (al least two), though the radial component is still present.

Coming back to the Euler equations, we choose $E_1=H^{s+1}(D)$, $E_0=H^{s}(D)$ and $E_{-1}=H^{s-1}(D)$, for the smooth, bounded and simply connected domain $D \subset \R^d$ ($d=2,3$) choosen at the beginning. In the case of the Euler equations, the SPDE reads as follow
\begin{align}
du_t = -\mathcal{P}(u_t\cdot \nabla) u_t\, dt +\sigma(X_t) \circ dW_t.\label{eq:Stoch_Euler_intro}
\end{align}
The reason why the Euler equations fall within the abstract framework of Theorem \ref{thm:main_intro} is that the drift satisfies condition \eqref{eq:drift_bound_intro}. In particular, for $m>d/2+1$ and $u\in H^{m+1}(D)$ 
\begin{align*}
  \langle \mathcal{P}[(u\cdot \nabla)u], u\rangle_{H^m} \le C\|u\|_{W^{1,\infty}}\|u\|^2_{H^m}\le C\|u\|^3_{H^m},
\end{align*}
where we exploited the Sobolev embedding in the last inequality. As a consequence, we can apply the abstract result obtaining the following theorem.

\begin{theorem}[informal, see Theorem \ref{thm:ApplicationToEuler}]
Let $\eta> 3/2$, $c_N>0$ (or alternatively $\eta =3/2$, $c_N$ large enough) and $\alpha$ satisfying \eqref{eq:alpha_condition_intro}. Moreover, let $s\in \mathbb{N}$ and $s>d/2+1$.
Then, for every initial condition $u_0\in H^{s+1}(D)$, the stochastic Euler equations~\eqref{eq:Stoch_Euler_intro} have a unique, global-in-time, strong solution in $L^\infty([0,T];H^{s+1}(D))\cap C([0,T];H^{s}(D))$, $\Prob$-a.s., for all $T>0$.
\end{theorem}

\subsection{Strategy of the proof and organization of the paper}

In Section \ref{section:Assumptions_&_Main_result} we introduce the notation, the abstract framework, the various assumptions and, in the end, we state the main result for the abstract SDE \eqref{eq:SDE_intro}.

In Section \ref{section:local_maximal_solution} we show existence and uniqueness of a local maximal solution of \eqref{eq:SDE_intro}. This is done without taking advantage of the specific shape of the noise. The construction of the solution is made by gluing together solutions to a localized SDE, obtained by multiplying the coeffiecients of \eqref{eq:SDE_intro} by a cutoff function. In particular, in Subsection \ref{subsection:Galerkin}, we consider finite-dimensional Galerkin approximations to the localized SDE and we show uniform bounds that imply the precompactness of such Galerkin approximations. In Subsection \ref{subsection:passage_to_the_limit} we pass to the limit, showing existence of weak solutions for the localized SDE. Finally, in Subsection \ref{subsecton:local_solution}, taking advantage of a uniqueness result, we are able to unambiguously construct local maximal solutions to \eqref{eq:SDE_intro}, by gluing together solutions of localized SDEs associated to different parameters of the cutoff function.

Section \ref{section:global_solution} is the core of the work. Here, we exhibit a suitable Lyapunov function, through which we are able to show global existence for \eqref{eq:SDE_intro}. The Lyapunov function is approximately $V(x)\approx \log \lVert x \rVert_{E_0}$ and its main property is the concavity along the radial direction. The concavity is responsible of a negative sign in the second derivative, which helps in computations due to It\^o lemma. In particular, it is possible to show a Gronwall estimate for $V(X_t)$, which ultimately implies the no blow-up result for $X_t$.

In Section \ref{section:Euler} we prove the no-blow up for $H^s$-solutions of the Euler equations ($s\in \N$, $s>d/2+1$). The result is achieved as an application of the no-blow up result for the abstract SDE \eqref{eq:SDE_intro}. 

We postponed in Appendix \ref{section:appendix_proofs} the proofs of some technical lemmas.

\subsection*{Acknowledgments}
The author would like to thank Mario Maurelli for the multiple pieces of advice and valuable discussions, which have been fundamental for the development of this note.

\section{Setting, assumptions and main result} \label{section:Assumptions_&_Main_result}

Let $E_1$, $E_0$, and $E_{-1}$ be three separable Hilbert spaces such that $E_1$ is compactly embedded in $E_0$ and $E_0$ is continuously embedded in $E_{-1}$, i.e.
\begin{align*}
E_1\subset \subset E_0\hookrightarrow E_{-1},
\end{align*}
and let $(\Omega,\mathcal{A},(\mathcal{F}_t)_t,\mathbb P)$ be a filtered probability space satisfying the standard assumption. We address the following stochastic differential equation,
\begin{align}
dX_t = b(X_t)dt +\sigma(X_t) \circ dW_t,\label{eq:SDE_Strat_compact}
\end{align}
where $b: E_0 \to E_{-1}$ is a Borel function, $\sigma: E_0 \to L(E_0,E_0)$ is a $C^1$ function, $W$ is a $Q$-Wiener process on $(\Omega,\mathcal{A},(\mathcal{F}_t)_t,\mathbb P)$ with values in $E_0$, for a given trace-class operator $Q:E_0\to E_0$, and $\circ$ denotes the Stratonovich stochastic integration.
By the compactness of trace class operators, the $Q$-Wiener process $W$ can be decomposed as
\begin{align*}
W=\sum_{k\ge1} \sqrt{\lambda_k}w_k W^k
\end{align*}
where $\{w_k\}_{k\in \N}$ is an orthonormal basis of $E_0$ made of eigenvectors of $Q$, $\{\lambda_k\}_{k\in\N}$ is the associated sequence of non-negative eigenvalues and $\{W^k\}_{k\in\N}$ are independent real Brownian motions on the filtered probability space $(\Omega,\mathcal{A},(\mathcal{F}_t)_t,\mathbb P)$.

\begin{assumption}\label{Assump:WienerProcess}
    We require that, despite being $W$ a $Q$-Wiener process in $E_0$, the eigenvectors $\{w_k\}_{k\in \N}$ associated to its covariance operator $Q$ are actually more regular and belong to $E_1$. Moreover, we require that $\{\lambda_k\}_{k\in\N}$ decay to $0$ fast enough such that
    \begin{equation}
        \sum_{k=1}^\infty \lambda_k \lVert w_k\rVert^2_{E_1}<\infty.
    \end{equation}
\end{assumption}

With the previous decomposition of the noise, the SDE \eqref{eq:SDE_Strat_compact} can be written as
\begin{align}\label{eq:SDE_Strat_extended}
dX_t = b(X_t)dt +\sum_{k\ge1} \sigma_k(X_t) \circ dW^k_t,
\end{align}
or, in It\^o form, as
\begin{align} \label{eq:SDE_Ito}
dX_t = b(X_t)dt +\frac12 \sum_{k\ge1} \left(\left(\sigma_k\cdot \nabla \right) \sigma_k\right)(X_t) dt+\sum_{k\ge1} \sigma_k(X_t) dW^k_t,
\end{align}
where
\begin{align}\label{eq:sigma_k}
\sigma_k(x)\coloneqq\sqrt{\lambda_k} \sigma(x) w_k.
\end{align}
For notational convenience, we will often write
\begin{align*}
&\left(\sigma\cdot\nabla\sigma\right)_k (x)\coloneqq \left(\left( \sigma_k\cdot \nabla \right)\sigma_k\right)(x),\\
&\left(\sigma\cdot\nabla\sigma\right)(x)\coloneqq \sum_{k\ge1} \left(\sigma\cdot\nabla\sigma\right)_k(x).
\end{align*}
For a given deterministic $E_1$-valued initial condition $X_0$, we give the definition of strong solution. We base our notion of solution on equation \eqref{eq:SDE_Ito} (i.e.\ the It\^o form). A posteriori, being $X$ a continuous semimartingale, it entails that $X$ satisfies also the corresponding Stratonovich equation \eqref{eq:SDE_Strat_extended}.

We remind that a stopping time $\tau$ is said to be accessible if there exists a non-decreasing sequence $(\tau_n)_{n\in\N}$ of stopping times, with $\lim_n \tau_n = \tau$ and $\tau_n<\tau$ for every $n$, $\Prob$-a.s.. Such a sequence $(\tau_n)_{n\in\N}$ is said to be an announcing sequence for $\tau$. Without loss of generality (possibly replacing $\tau_n$ with $\tau_n\wedge n$), we can and we will assume that $\tau_n$ is finite for every $n$.

\begin{definition}\label{def:strong}
Let $(\Omega,\mathcal{A},(\mathcal{F}_t)_t,\mathbb P)$ be a probability space and let $W$ be a $E_0$-valued $Q$-Wiener process. Let $\tau$ be an accessible stopping time, with announcing sequence $(\tau_n)_{n\in\N}$. We say that a progressively measurable process $(X_t)_{t\in [0,\tau)}$ on $(\Omega,\mathcal{A},(\mathcal{F}_t)_t,\mathbb P)$ and with values in $E_0$ is a strong solution to the equation \eqref{eq:SDE_Ito}, with the initial condition $X_0$, if 
\begin{itemize}
    \item $\Prob$-a.s., $X$ belongs to $C([0,\tau);E_0)$ and to $L^\infty([0,\tau_n);E_1)$, for every $n$;
    \item $\Prob$-a.s., $t\mapsto b(X_t)$, $t\mapsto\left(\sigma\cdot\nabla\sigma\right)(X_t)$ and $t\mapsto \sum_{k\ge1} \lVert \sigma_k(X_t) \rVert^2_{E_0}$ are locally integrable functions on $[0, \tau)$;
    \item the following identity holds in $E_{0}$  $\Prob$-a.s.:
\begin{align}\label{eq:SDE_Ito_Integral}
X_t = X_0 + \int_0^t b(X_r) dr +\frac12\int_0^t\left(\sigma\cdot \nabla  \sigma \right)(X_r) dr+\sum_{k\ge1} \int_0^t\sigma_k(X_r) dW^k_r,\quad \forall t\in [0,\tau). 
\end{align}
\end{itemize}
We say that $X$ is a global strong solution if $\tau=\infty$ $\Prob$-a.s..
\end{definition}

\begin{remark}
We assume $b$ to be $E_{-1}$-valued rather than $E_0$-valued to include fluid dynamic models whose drifts depend on the spatial gradient of $X_t$.
\end{remark}

In the remaining section, we summarize the assumptions needed for the main theorem and we conclude the section with its statement.
\begin{assumption}\label{Assump:Spaces}
 For separable Hilbert spaces $E_{-1},E_0,E_1$ introduced at the beginning, the compact embedding $E_1\subset \subset E_0$ and the continuous embedding $E_0\hookrightarrow E_{-1}$ are both dense. Moreover, there exist $\theta\in (0,1)$ and $M>0$ such that $\|v\|_{E_0}\le M\|v\|_{E_1}^{1-\theta}\|v\|_{E_{-1}}^\theta$ for all $v\in E_1$.
\end{assumption}

\begin{remark} \label{remark:norm_semi-continuity}
    As a consequence of Assumption \ref{Assump:Spaces}, the norm function $\lVert\cdot\rVert_{E_1}: E_0 \to [0,+\infty]$ is Borel and lower semi-continuous. Indeed, given a dense countable set $\left(f_n\right)_{n \in \N} \subset E_1^*$, we can write
    \begin{equation*}
        \lVert x \rVert_{E_1} = \sup_{n\in N} \frac{1}{\lVert f_n \rVert_{E_1^*}}f_n(x).
    \end{equation*}
    Being the embedding $E_0^*\hookrightarrow E_1^*$ continuous and dense, we can choose $f_n$ to be the functional induced by a countable dense set $\left(x_n\right)_{n \in \N} \subset E_0$. In particular, with slight abuse of notation, 
    \begin{equation*}
        \lVert x \rVert_{E_1} = \sup_{n\in N} \frac{1}{\lVert x_n \rVert_{E_1^*}}\langle x_n, x\rangle_{E_0}.
    \end{equation*}
    Therefore $\lVert\cdot\rVert_{E_1}: E_0 \to [0,+\infty]$ is Borel and lower semi-continuous, as a supremum of countably many lower semi-continuous functions.
\end{remark}

The construction of the solution of \eqref{eq:SDE_Ito} will make use of Galerkin projections within an approximation argument. The Galerkin projections will be performed with respect to a vector basis specified by the following assumption.

\begin{assumption}\label{Assump:Projection}
There exists a sequence of elements $\{e_n\}_{n\in\mathbb{N}_+}$ in $E_1$ whose span is dense in $E_{-1}$ such that, if $\Pi_n: E_{-1} \to E^n$ denotes the orthogonal projector from $E_{-1}$ to its subspace $E^n\coloneqq\text{span}\{e_1,\ldots e_n\}$, then for all $x \in E_1$
\begin{equation*}
\|\Pi_n x\|_{E_1}\le C\|x\|_{E_1},
\end{equation*}
for a constant $C>0$ independent of $n\in\mathbb{N}_+$ and $x$.
\end{assumption}

\begin{remark}
Assumption \ref{Assump:Projection} is not strictly required. Indeed, using Assumption \ref{Assump:Spaces}, one can always proceed as in the proof of Lemma~\ref{Euler:basis} (see below) to construct a basis of $E_{-1}$ that is made of $E_{1}$ elements and is orthogonal in terms of both scalar products. (In this situation, we can choose $C=1$.) The reason why we impose Assumption \ref{Assump:Projection} is that the choice of $\{e_n\}_{n\in\mathbb{N}^+}$ also enters Assumptions \ref{Assump:drift} and \ref{Assump:sigma}.
\end{remark}

Now we introduce, for $x\in E^n$,
\begin{align}
\label{eq:b^n&sigma_n}
\begin{split}
&b^n(x) = \Pi_n b(x),\\
&\left(\sigma \cdot\nabla\sigma\right)^n_k(x) = \Pi_n \left(\sigma \cdot\nabla\sigma\right)_k(x),\\
&\left(\sigma \cdot\nabla\sigma\right)^n(x) = \sum_{k=1}^n \left(\sigma \cdot\nabla\sigma\right)^n_k(x),\\
&\sigma^n(x) = \left(\Pi_n \sigma_1(x), \,\cdots, \Pi_n \sigma_n(x) \right).
\end{split}
\end{align}
 
To state the assumptions on the drift~$b$ and the noise coefficient~$\sigma$, we denote the open ball centered at the origin with radius $R$ by $B_R$ and the complement by $B_R^c$. If there is a need to emphasize the underlying topology, we write $B_{R;\mathcal{H}}$ and $B^c_{R;\mathcal{H}}=\mathcal{H}\setminus B_{R;\mathcal{H}}$ (i.e., $B_{R;\mathcal{H}}$ is a subset of $\mathcal{H}$, whose radius is measured in the $\mathcal{H}$-norm and $B^c_{R;\mathcal{H}}$ is the complement of $B_{R;\mathcal{H}}$ in $\mathcal{H}$). The closure of the ball is denoted by $\overline{B}_{R}$ or $\overline{B}_{R;\mathcal H}$.

We say that a function $f:\mathcal{H}\to \mathcal{K}$, with $\mathcal{H},~\mathcal{K}$ being separable Hilbert spaces, is bounded on balls if $\sup_{x\in \overline{B}_{R;\mathcal H}}\|f(x)\|_\mathcal{K}<\infty$ for every $R>0$. We use $a\vee b$ for $\max\{a,b\}$.

\begin{assumption}\label{Assump:drift} We assume that:
\begin{itemize}
\item[a)] $b:E_0\to E_{-1}$ is continuous and bounded on balls and $b(E_1)\subseteq E_0$.
\item[b)] There exist $R\ge 1$, $m > 2$, $c_D>0$ and a non-decreasing $g:\R\to \mathbb{R}$, such that
\begin{align*} 
\begin{cases}
\langle b^n(x), x \rangle_{E_1} \le C+g(\lVert x\rVert_{E_0})\lVert x\rVert_{E_1}^{2} \qquad& \forall x\in E^n, \;\forall  n\in \mathbb{N},\\
\langle b(x),x \rangle_{E_0} \le g(\lVert x\rVert_{E_0}) \lVert x \rVert^2_{E_0} &\forall x \in B_{R;E_0}^c \cap E_1, \\
g(\lVert x\rVert_{E_0})\le c_D\lVert x \rVert^{m-2}_{E_0} &\forall  x \in B_{R;E_0}^c \cap E_1.
\end{cases}
\end{align*}
\item[c)] The projected drift~$b^n:E^n\to E^n$ is locally Lipschitz for every $n\in \mathbb{N}$.
\item[d)] There exists an increasing function $\chi:\mathbb{R}\to \mathbb{R}_{+}$ such that
\begin{align*}
&\langle b(x)-b(y),x-y\rangle_{E_0} \leq
\chi (\Vert x\Vert_{E_1}\vee\Vert y\Vert_{E_1} )\Vert x-y\Vert_{E_0}^2\quad \text{for all } x,y\in E_{0}\cap E_1.
\end{align*}
\end{itemize}
\end{assumption}

\begin{assumption} \label{Assump:sigma} We assume that:
\begin{itemize}
\item[a)] $\sigma:E_0\to L(E_0,E_0)$ is $C^2$ with $\sigma$, $D\sigma:E_0\to L(E_0,L(E_0,E_0))$ and $D^2\sigma:E_0\to L(E_0,L(E_0,L(E_0,E_0)))$ bounded on $E_0$-balls.
\item[b)] There exists an increasing function $f:\mathbb{R}\to \mathbb{R}_+$ such that
\begin{align*}
\begin{cases}
\langle x, \left(\sigma\cdot \nabla \sigma\right)^n(x)\rangle_{E_1} \leq
f (\Vert x\Vert_{E_0} ) \lVert x \rVert_{E_1}^2 \quad& \forall x\in E^n, \;\forall  n\in \mathbb{N},\\
\sum_{k=1}^n \Vert   \sigma^n_k(x)\Vert^2_{E_1} \leq
f (\Vert x\Vert_{E_0} ) \lVert x \rVert_{E_1}^2 \quad& \forall x\in E^n, \;\forall  n\in \mathbb{N}.
\end{cases}
\end{align*}
\item[c)] We specify a particular form for $\sigma(x)$ far from the origin, i.e.
\begin{align} \label{eq:sigma}
    \sigma(x) = c_N\lVert x \rVert^\eta_{E_0} \left( I - \alpha  \frac {x \otimes x}{\lVert x \rVert^2_{E_0}} \right) \qquad \forall x \in B_{R;E_0}^c,
\end{align}
where both the norms and the implicit scalar product in $x \otimes x = x \langle x, \cdot \rangle$ are intended in $E_0$, $\alpha,\eta \in \R$ and $c_N>0$. We also prescribe
\begin{align*}
\eta \ge m/2
\end{align*}
for the $m$ introduced in Assumption \ref{Assump:drift}-(b).
\end{itemize}
\end{assumption}

Under Assumption \ref{Assump:sigma}, for $x\in B^c_{R;E_0}$, the noise coefficients $\sigma_k$ defined in \eqref{eq:sigma_k} take the form
\begin{align*}
    \sigma_k(x)=c_N\sqrt{\lambda_k} \lVert x \rVert^\eta \left( w_k - \alpha  \frac {1}{\lVert x \rVert^2}\langle x, w_k\rangle x\right),
\end{align*}
where, again, scalar products and norms are intended in $E_0$. After some computations, for $X_t\in B^c_{R;E_0}$, the SDE appears in the following form
\begin{align}\label{eq:SDE_ItoForm}
dX_t = b(X_t)dt +\frac12 \sum_{k\ge1} \Bigl(\alpha_k(X_t)w_k +\beta_k(X_t) X_t \Bigr) dt+\sum_{k\ge1}\Bigl( \gamma_k(X_t)w_k + \delta_k(X_t) X_t \Bigr)dW^k_t,
\end{align}
where $\alpha_k,\beta_k,\gamma_k,\delta_k: E_0 \to \R$ are given by
\begin{align}\label{eq:alpha_beta_gamma_delta}
\begin{split}
    \alpha_k(x)=&c_N^2\lambda_k (\eta - \alpha - \eta\alpha)\lVert x \rVert^{2\eta-2}_{E_0} \langle x,w_k\rangle_{E_0}\\
    \beta_k(x)=&c_N^2\lambda_k \lVert x \rVert^{2\eta-2}_{E_0}\Bigl(\lVert x \rVert^{-2}_{E_0} \langle x,w_k\rangle^2_{E_0}(\alpha^2\eta-\alpha(\eta-2)) -\alpha\lVert w_k \rVert_{E_0}^2\Bigr)\\
    \gamma_k(x)=&c_N\sqrt{\lambda_k} \lVert x \rVert^\eta_{E_0} \\
    \delta_k(x)=& -c_N \sqrt{\lambda_k}\alpha \lVert x \rVert^{\eta-2}_{E_0}\langle x,w_k\rangle_{E_0}.
\end{split}
\end{align}

\begin{lemma} \label{lemma:noise_properties}
    For $\sigma$ satisfying Assumption \ref{Assump:sigma}-(a) the following properties hold.\\ 
    There exists an increasing function $h:\mathbb{R}\to \mathbb{R}_+$ such that
            \begin{align}
            &\label{eq:sigma_property_1}
            \sum_{k\ge 1} \Vert  \sigma_k(x)\Vert^2_{E_0} \leq
            h (\Vert x\Vert_{E_0} )\\
            &\label{eq:sigma_property_2}
            \sum_{k\ge 1} \Vert   \sigma_k(x)-\sigma_k(y)\Vert^2_{E_0} \leq
            h (\Vert x\Vert_{E_0}\vee\Vert y\Vert_{E_0} )\Vert x-y\Vert^2_{E_0}\\
            &\label{eq:sigma_property_3}
            \sum_{k\ge 1}\sup_{x\in B_{R;E_0}}\|[\sigma\cdot\nabla \sigma]_k(x)\|_{E_0} \le h(R)\\
            &\label{eq:sigma_property_4}
            \Big\Vert \left(\sigma\cdot \nabla \sigma\right)(x)  \Big\Vert_{E_0} \le h (\Vert x\Vert_{E_0})\\
            &\label{eq:sigma_property_5}
            \Big\Vert \left(\sigma\cdot \nabla \sigma\right)(x) - \left(\sigma\cdot \nabla \sigma\right)(y)\Big\Vert_{E_0} \le h (\Vert x\Vert_{E_0}\vee\Vert y\Vert_{E_0} )\Vert x-y\Vert_{E_0}
            \end{align}
    Moreover, $\forall n\in\mathbb{N}$, $\sigma^n:E^n\to (E^n)^n$ and $\left(\sigma \cdot\nabla\sigma\right)^n:E^n\to E^n$ are locally Lipschitz.
\end{lemma}
\begin{proof}
    Proof in the appendix, see Subsection \ref{subsection:proof_lemma_noise_properties}.
\end{proof}

\begin{lemma} \label{lemma:coherence}
    The prescribed form for $\sigma$ given in \eqref{eq:sigma} is consistent with the other assumptions in Assumption \ref{Assump:sigma}. In particular, if $\psi :\R \to [0,1]$ is a smooth even function such that $\psi(s)\equiv 0$ for $|s|\le R/2$ and $\psi(s)\equiv 1$ for $|s|\ge R$,
\begin{align} \label{eq:sigma_full_space}
    \sigma(x) = c_N\psi (\lVert x \rVert_{E_0})\lVert x \rVert^\eta_{E_0} \left( I - \alpha  \frac {x \otimes x}{\lVert x \rVert^2_{E_0}} \right)
\end{align}
is an example of $\sigma$ defined on the whole space $E_0$ and satisfying Assumption \ref{Assump:sigma}.
\end{lemma}
\begin{proof}
    Proof in the appendix, see Subsection \ref{subsection:proof_lemma_coherence}.
\end{proof}

Below is our first main result of this paper, which is an immediate conclusion of Theorems~\ref{thm:global_existence} and~\ref{thm:pathwise_uniqueness} in Section \ref{section:global_solution}.

\begin{theorem} \label{thm:main}
Under Assumptions \ref{Assump:WienerProcess}, \ref{Assump:Spaces}, \ref{Assump:Projection}, \ref{Assump:drift} and \ref{Assump:sigma}, let in \eqref{eq:sigma} $\eta> m/2$, $c_N>0$ (or alternatively $\eta =m/2$, $c_N$ large enough) and $\alpha$ such that
\begin{align}\label{eq:alpha_main}
    1 < \alpha < 1+ \frac{1}{\eta-1} \frac{\Tr Q - \lVert Q \rVert}{\lVert Q \rVert}.
\end{align} 
Then, for every deterministic initial condition $x_0 \in E_1$, there exists a unique, global-in-time, strong solution of \eqref{eq:SDE_intro} in $L^\infty([0,T];E_1)\cap C([0,T];E_0)$, $\Prob$-a.s., for all $T>0$.
\end{theorem}

The rigorous statement for the Euler equations, i.e.\ Theorem \ref{thm:ApplicationToEuler}, is an application of Theorem \ref{thm:main}.

\section{Local well-posedness for maximal solutions} \label{section:local_maximal_solution}
In this section, we show existence and uniqueness for a local maximal solution of \eqref{eq:SDE_Ito}. Here, we do not take advantage of the specific shape of the noise. The construction of the solution is made in multiple steps. For $M>0$, we introduce a smooth non-increasing cutoff function $\theta_M:[0,\infty) \to [0,1]$ such that
\begin{align*}
    \theta_M (s) =\begin{cases}
        1 &\text{for }s<M \\
        0 & \text{for }s>2M.
    \end{cases}
\end{align*}
As a way to localize the equation, we multiply both drift and diffusion coefficients in \eqref{eq:SDE_Ito} by $\theta_M(\lVert \cdot \rVert_{E_0})$, obtaining a class of cutoff SDEs parametrized by $M$
\begin{align}\label{eq:SDE_Ito_cutoff}
\begin{split}
dX^{M}_t = &\theta_M(\lVert X^{M}_t \rVert_{E_0})\left( b(X^M_t) +\frac12 \left(\sigma\cdot\nabla\sigma\right)( X^M_t) \right)dt +\theta_M(\lVert X^{M}_t \rVert_{E_0})\sum_{k\ge1}\sigma_k(X^M_t)dW^k_t.
\end{split}
\end{align}
The notion of probabilistically weak and strong solutions to \eqref{eq:SDE_Ito_cutoff} is completely analogous to Definition \ref{def:strong}, meaning that we require solutions to have paths in $C([0,T);E_0)\cap L^\infty([0,T);E_1)$ $\Prob$-a.s..
In Subsection \ref{subsection:Galerkin}, we consider finite-dimensional Galerkin approximations of \eqref{eq:SDE_Ito_cutoff} and we show uniform bounds that imply the precompactness of such Galerkin approximations. In Subsection \ref{subsection:passage_to_the_limit} we pass to the limit, showing existence of weak solutions for \eqref{eq:SDE_Ito_cutoff}. Finally, in Subsection \ref{subsecton:local_solution}, taking advantage of a uniqueness result for solutions of \eqref{eq:SDE_Ito_cutoff}, we are able to unambiguously construct local maximal solutions to \eqref{eq:SDE_Ito}, by gluing solution of \eqref{eq:SDE_Ito_cutoff} associated to different values of $M$.

\subsection{Precompactness of approximate solutions with cutoff} \label{subsection:Galerkin}

In this subsection, we show precompactness for a family of Galerkin approximations of the equation with cutoff.

For fixed $M>0$, consider the following finite-dimensional Galerkin approximations with cutoff on $E^n$ (see Assumption \ref{Assump:Projection}) to the SDE~\eqref{eq:SDE_Ito_cutoff}, 

\begin{equation} \label{eq:SDE_Galerkin}
    \begin{cases}
    dX^{n,M}_t=\theta_M(\lVert X^{n,M}_t \rVert_{E_0})\hat b^n(X^{n,M}_t)dt+\theta_M(\lVert X^{n,M}_t \rVert_{E_0})\sigma^n(X^{n,M}_t)dW^{(n)}_t, \\
    X^{n,M}_0=\Pi_n x_0, 
    \end{cases}
\end{equation}
where
\begin{align} \label{eq:Galerkin_drift}
    \hat b^n(x)&= b^n(x) + \frac 12 \left(\sigma\cdot\nabla\sigma\right)^n(x)
\end{align}
and $b^n$, $\left( \sigma \cdot \nabla \sigma \right)^n$ and $\sigma^n$ are defined in \eqref{eq:b^n&sigma_n}, $W^{(n)}=(W^1,\,\cdots, W^n)^T$ and $x_0\in E_1$ is deterministic. Since $\hat b^n, \sigma^n$ are locally Lipschitz (by Assumptions \ref{Assump:drift}-(c) and Lemma \ref{lemma:noise_properties}), the local existence of a unique solution $X^{n,M}_t$ to~\eqref{eq:SDE_Galerkin} is a classical result. Moreover, thanks to the cutoff function and the equivalence of norms in finite dimensional spaces, the solutions are actually global in time.
\begin{lemma} \label{Lemma:Uniform_Infinity_Bound}
    For every $T>0$ and $M>0$, there exists $C_{T,M}>0$ such that
\begin{equation}\label{eq:Uniform_Infinity_Bound}
    \sup_n \E \sup_{t\in[0,T]}  \lVert X^{n,M}_t \rVert_{E_1}^2 \le C_{T,M}.
\end{equation}
\end{lemma}
\begin{proof}[Proof of Lemma~\ref{Lemma:Uniform_Infinity_Bound}]
For ease of applying It\^o's lemma, we endow the $E_n$-space with the $E_1$-scalar product for every $n\in \mathbb{N}$. By the It\^o formula, $\lVert X^{n,M} \rVert_{E_1}^2$ has the following stochastic differential 
\begin{align}\label{eq:Energy_differential}
\begin{split}
    d\lVert X^{n,M}_t \rVert^2_{E_1}=
    & \,2\theta_M(\lVert X^{n,M}_t \rVert_{E_0})\langle X^{n,M}_t, b^n(X^{n,M}_t)\rangle_{E_1} dt\\
    & + \theta_M(\lVert X^{n,M}_t \rVert_{E_0}) \langle X^{n,M}_t, \left(\sigma\cdot \nabla \sigma\right)^n(X^{n,M}_t) \rangle_{E_1} dt\\
    &+ \theta_M(\lVert X^{n,M}_t \rVert_{E_0})\sum_{k=1}^n  \rVert \sigma^n_k(X^{n,M}_t) \rVert^{2}_{E_1}dt\\
    &+ 2\theta_M(\lVert X^{n,M}_t \rVert_{E_0})\sum_{k=1}^n \langle X^{n,M}_t, \sigma^n_k \rangle_{E_1} dW^k_t.
\end{split}
\end{align}
We write \eqref{eq:Energy_differential} in integral form, we take the supremum in time and expectation and we analyse the right hand side term by term. 
For the first term in \eqref{eq:Energy_differential}, we exploit Assumption \ref{Assump:drift}-(b), obtaining
\begin{align}\label{eq:first_estimate}
\begin{split}
    \E\sup_{t\in[0,S]} &\int_0^t 2\theta_M(\lVert X^{n,M}_r\rVert_{E_0})\langle X^{n,M}_r, b^n(X^{n,M}_r)\rangle_{E_1}\,dr\\ 
    &\le \E\int_0^S 2\theta_M(\lVert X^{n,M}_r\rVert_{E_0}) (C+g(\lVert X^{n,M}_r\rVert_{E_0}) \lVert X^{n,M}_r \rVert_{E_1}^2) \,dr\\
    &\le CS +C_M\E \int_0^S \lVert X^{n,M}_r \rVert_{E_1}^2 \,dr \le CS  +C_M\int_0^S \E \sup_{t\in[0,r]}\lVert  X^{n,M}_t \rVert_{E_1}^2\,dr.
\end{split}
\end{align}
In the second term and third term in \eqref{eq:Energy_differential}, exploiting Assumptions \ref{Assump:sigma}-(b), we get
\begin{align} \label{eq:second_estimate}
\begin{split}
    \E\sup_{t\in[0,S]} &\int_0^t \theta_M(\lVert X^{n,M}_r \rVert_{E_0})  \left(\langle X^{n,M}_r, \left(\sigma\cdot \nabla \sigma\right)^n(X^{n,M}_r) \rangle_{E_1} + \sum_{k=1}^n\rVert \sigma^n_k(X^{n,M}_r) \rVert^{2}_{E_1} \right)\,dr\\ 
    &\le 2 \,\E\int_0^S \theta_M(\lVert X^{n,M}_r \rVert_{E_0}) f(\lVert X^{n,M}_r \rVert_{E_0}) \lVert X^{n,M}_r \rVert^2_{E_1}  \,dr\\
    &\le  C_M \E \int_0^S \lVert X^{n,M}_r \rVert_{E_1}^2 \,dr \le C_M  \int_0^S \E \sup_{t\in[0,r]}\lVert  X^{n,M}_t \rVert_{E_1}^2\,dr.
\end{split}
\end{align}
Finally, to control the stochastic integral, we employ Burkholder-Davis-Gundy inequality, Assumption \ref{Assump:sigma}-(b) and Young inequality
\begin{align} \label{eq:last_estimate}
\begin{split}
    \E&\sup_{t\in[0,S]} \int_0^t 2\theta_M(\lVert X^{n,M}_r \rVert_{E_0})\sum_{k=1}^n \langle X^{n,M}_r, \sigma_k( X^{n,M}_r ) \rangle_{E_1} dW^k_r \\
    &\le C\E \left( \int_0^S \theta^2_M(\lVert X^{n,M}_r \rVert_{E_0})\sum_{k=1}^n \langle X^{n,M}_r, \sigma_k( X^{n,M}_r ) \rangle_{E_1}^2 dr\right)^{\frac12}\\
    &\le C\E \left( \int_0^S \theta^2_M(\lVert X^{n,M}_r \rVert_{E_0}) \lVert X^{n,M}_r \rVert^2_{E_1}\sum_{k=1}^n \lVert \sigma_k( X^{n,M}_r )\rVert^2_{E_1}   dr\right)^{\frac12}\\
    &\le C\E \left( \int_0^S \theta^2_M(\lVert X^{n,M}_r \rVert_{E_0}) f(\lVert  X^{n,M}_r \rVert_{E_0} ) \lVert X^{n,M}_r \rVert^4_{E_1}  dr\right)^{\frac12}\\
    &\le C_M\E \left( \int_0^S  \lVert X^{n,M}_r \rVert^4_{E_1}  dr\right)^{\frac12}\\
    &\le C_M\E \left( \sup_{t\in[0,S]}\lVert  X^{n,M}_t \rVert_{E_1}^2\int_0^S \lVert X^{n,M}_r \rVert^2_{E_1} dr\right)^{\frac12}\\
    &\le \frac12 \E \sup_{t\in[0,S]}\lVert  X^{n,M}_t \rVert_{E_1}^2 + C_M\E \int_0^S  \lVert X^{n,M}_r \rVert^2_{E_1} dr\\
    &\le \frac12 \E \sup_{t\in[0,S]}\lVert  X^{n,M}_t \rVert_{E_1}^2 + C_M \int_0^S \E \sup_{t\in[0,r]}\lVert  X^{n,M}_t \rVert_{E_1}^2\,dr.
\end{split}
\end{align}
Denoting $Y_t^{n,M}\coloneqq \E \sup_{s\in[0,t]}\lVert  X^{n,M}_s \rVert_{E_1}^2$ and combining estimates \eqref{eq:first_estimate}-\eqref{eq:last_estimate}, we end up with
\begin{equation*}
    Y_t^{n,M}\le Y_0^{n,M} +Ct + C_M \int_0^t Y_r^{n,M} \, dr.
\end{equation*}
Therefore, by Gronwall inequality and Assumption \ref{Assump:Projection},
\begin{align*}
     \E \sup_{s\in[0,T]}\lVert  X^{n,M}_s \rVert_{E_1}^2 = Y_T^{n,M}\le \left( Y_0^{n,M}+ CT\right) e^{C_M T} = \left( \E \lVert  X^{n,M}_0 \rVert_{E_1}^2+ CT\right) e^{C_M T}\le C_{M,T}
\end{align*}
which conclude the proof.
\end{proof}

The next result provides a uniform bound for $X^{n,M}$ in $C^{0,\alpha}([0,T];E_{-1})$ with $\alpha\in (0,1/2)$. Here, $C^{0,\alpha}([0,T];E_{-1})$ refers to the H\"{o}lder spaces defined by the norm
\begin{align*}
&\|f\|_{C^{0,\alpha}(E_{-1})} := \|f\|_{C(E_{-1})} + [f]_{C^{0,\alpha}(E_{-1})} := \sup_{t\in[0,T]}\|f(t)\|_{E_{-1}} 
+
\sup_{t,s\in[0,T]; t\neq s}\frac{\lVert f(t) -f(s)\rVert_{E_{-1}}}{|t-s|^{\alpha}}.
\end{align*}

\begin{lemma} \label{Lemma:Uniform_Holder_Bound}
Let $\alpha\in (0,1/2)$. For every $T>0$ and $M>0$, there exists $C_{T,M}>0$ such that
\begin{equation}\label{eq:Uniform_Holder_Bound}
\sup_n\E \lVert X^{n,M} \rVert_{C^{0,\alpha}(E_{-1})} \le C_{T,M}. 
\end{equation}
\end{lemma}

For the proof, we recall the Sobolev embedding $W^{\beta,p}\hookrightarrow C^{0,\alpha}$, where $1<p<\infty$ and  $\alpha<\beta-1/p$. More precisely, we may apply e.g.~\cite[Theorem B.1.5]{da1996ergodicity} and conclude that 
\begin{equation}
[f]_{C^{0,\alpha}(E_{-1})}^p \le C[f]_{W^{\beta,p}(E_{-1})}^p := C\int_0^T\int_0^T \frac{\|f_t-f_s\|_{E_{-1}}^p}{|t-s|^{1+\beta p}}\,ds\,dt \label{eq:Sobolev_emb_time}
\end{equation}
for a continuous function $f:[0,T]\to E_{-1}$, $1<p<\infty$, and $\alpha<\beta-1/p$.
\begin{proof}
Fix $p>1$. By Jensen inequality
\begin{align*}
    \E \lVert X^{n,M} \rVert_{C^{0,\alpha}(E_{-1})}&=\E \lVert X^{n,M} \rVert_{C(E_{-1})}+\E [X^{n,M} ]_{C^{0,\alpha}(E_{-1})} \\
    &\le \left(\E \lVert X^{n,M} \rVert^2_{C(E_{-1})}\right)^\frac12+\left(\E [X^{n,M} ]^p_{C^{0,\alpha}(E_{-1})}\right)^\frac1p.
\end{align*}
Therefore, thanks to Lemma \ref{Lemma:Uniform_Infinity_Bound} and the continuous embedding of $E_1$ into $E_{-1}$ (Assumption \ref{Assump:Spaces}), it is sufficient to prove
\begin{align*}
    \sup_n\E [X^{n,M} ]^p_{C^{0,\alpha}(E_{-1})} \le C_{T,M}.
\end{align*}
Exploiting \eqref{eq:Sobolev_emb_time}, we are left to show that
\begin{align} \label{eq:Uniform_Fractional_Sobolev}
    \sup_n\E \int_0^T\int_0^T \frac{\|X^{n,M}_t-X^{n,M}_s\|_{E_{-1}}^p}{|t-s|^{1+\beta p}}\,ds\,dt = \sup_n\int_0^T\int_0^T \frac{\E\|X^{n,M}_t-X^{n,M}_s\|_{E_{-1}}^p}{|t-s|^{1+\beta p}}\,ds\,dt \le C_{T,M}.
\end{align}
By writing \eqref{eq:SDE_Galerkin} in integral form, taking the $E_{-1}$ norm and the $p$-power and then expectation, we obtain
\begin{align}
\begin{split}\label{eq:time_cont_bd}
    \E&\|X^{n,M}_t-X^{n,M}_s\|_{E_{-1}}^p \\
    &\le C_p \E \left\Vert \int_s^t \theta_M(\lVert X^{n,M}_r \rVert_{E_0}) b^n(X^{n,M}_r) \,dr \right\Vert^p_{E_{-1}}\\
    &\quad + C_p \E \left\Vert \int_s^t \theta_M(\lVert X^{n,M}_r \rVert_{E_0}) \sum_{k=1}^n \left(\sigma\cdot \nabla\sigma\right)^n_k\left(X^{n,M}_r\right) \,dr \right\Vert^p_{E_{-1}}\\
    &\quad+C_p \E \left\Vert \int_s^t \theta_M(\lVert X^{n,M}_r \rVert_{E_0})\sum_{k=1}^n \sigma^n_k(X^{n,M}_r) dW^k_r \right\Vert^p_{E_{-1}}.
\end{split}
\end{align}
We analyse the above right hand side term by term. In the first term in the right hand side of \eqref{eq:time_cont_bd}, exploiting the orthogonality in $E_{-1}$ of the $\Pi_n$ projection and the boundedness of $b$ on $E_0$ balls,
\begin{align*}
    \E \left\Vert \int_s^t \theta_M(\lVert X^{n,M}_r \rVert_{E_0}) b^n(X^{n,M}_r) \,dr \right\Vert^p_{E_{-1}}
    &=\E \left\Vert \int_s^t \theta_M(\lVert X^{n,M}_r \rVert_{E_0})\Pi_n b(X^{n,M}_r) \,dr \right\Vert^p_{E_{-1}}\\
    &\le \E \left( \int_s^t \theta_M(\lVert X^{n,M}_r \rVert_{E_0})\lVert b(X^{n,M}_r)\rVert_{E_{-1}} \,dr \right)^p\\
    &\le C_M |t-s|^p.
\end{align*}
In the second term in the right hand side of \eqref{eq:time_cont_bd}, thanks to the orthogonality of $\Pi_n$ and \eqref{eq:sigma_property_3}, we have 
\begin{align*} 
    \E &\left\Vert \int_s^t \theta_M(\lVert X^{n,M}_r \rVert_{E_0}) \sum_{k=1}^n \left(\sigma\cdot \nabla\sigma\right)^n_k\left(X^{n,M}_r\right) \,dr \right\Vert^p_{E_{-1}}\\
    &\qquad=\E \left\Vert \int_s^t \theta_M(\lVert X^{n,M}_r \rVert_{E_0}) \sum_{k=1}^n \Pi_n \left(\sigma\cdot \nabla\sigma\right)_k\left(X^{n,M}_r\right) \,dr \right\Vert^p_{E_{-1}}\\
    &\qquad\le\E \left( \int_s^t \theta_M(\lVert X^{n,M}_r \rVert_{E_0}) \sum_{k=1}^\infty \left\Vert \left(\sigma\cdot \nabla\sigma\right)_k\left(X^{n,M}_r\right) \right\Vert_{E_{-1}} \,dr\right)^p \\
    &\qquad \le C_M |t-s|^p.
\end{align*}
Lastly, for the stochastic integral term in \eqref{eq:time_cont_bd} we employ Burkolder-Davis-Gundy inequality and exploit the fact that $\sigma$ is bounded on $E_0$ balls and $Q$ is trace-class
\begin{align*}
    &\E \left\Vert \int_s^t \theta_M(\lVert X^{n,M}_r \rVert_{E_0})\sum_{k=1}^n \sigma^n_k(X^{n,M}_r) dW^k_r \right\Vert^p_{E_{-1}}\\ 
    &\qquad\qquad= \E \left\Vert \int_s^t \theta_M(\lVert X^{n,M}_r \rVert_{E_0})\sum_{k=1}^n \Pi_n\sigma_k(X^{n,M}_r) dW^k_r \right\Vert^p_{E_{-1}}\\
    &\qquad\qquad\le C_p \,\E \left( \int_s^t \theta^2_M(\lVert X^{n,M}_r \rVert_{E_0})\sum_{k=1}^\infty \lVert\sigma_k(X^{n,M}_r)\rVert^2_{E_{-1}} \,dr \right)^\frac p2\\
    & \qquad\qquad \le C_{p,M} |t-s|^\frac p2.
\end{align*}
Combining all estimates, we end up with 
\begin{align*}
    \E&\|X^{n,M}_t-X^{n,M}_s\|_{E_{-1}}^p \le C_{p,M,T} |t-s|^\frac p2
\end{align*}
which implies the uniform bound in \eqref{eq:Uniform_Fractional_Sobolev} as soon as $\beta<\frac12$. Given $\alpha\in (0,1/2)$, we can always choose $\beta, p$ so that $\beta<1/2$, $1<p<\infty$, and $\alpha<\beta-1/p$. Therefore, the Sobolev embedding completes the proof of the lemma.
\end{proof}

\subsection{Global existence for the equation with cutoff} \label{subsection:passage_to_the_limit}

In this subsection, we show that every limit point of the Galerkin approximations converge to a solution to the equation with cutoff, thus providing weak existence for such equation. 

Recall the Aubin-Lions-Simon theorem (see~\cite[Corollary 9]{Si86}), which shows that, under the Assumption~\ref{Assump:Spaces} when $\alpha\in (0,1)$,
\[
L^\infty([0,T];E_1)\cap C^{0, \alpha}([0,T];E_{-1})
\subset\subset C([0,T];E_0).
\] 
Then, the set of the laws of $X^{n,M}$ is tight in $C([0,T];E_0)$ as a consequence of Lemmas \ref{Lemma:Uniform_Infinity_Bound} and \ref{Lemma:Uniform_Holder_Bound}, and then the set of laws $\{(X^{n.M},W )\}_{n\in\mathbb{N}}$ is tight in the product Polish space $C\left([0,T];E_0\right)\times C\left([0,T];E_0\right)$, as a consequence of the tightness of the two components. Applying Prokhorov's theorem, we conclude that the set of the laws $\{(X^{n,M},W )\}_{n\in\mathbb{N}}$ is a precompact set in the topology of weak convergence.

We will show that a weak limit point of $\{(X^{n,M},W )\}_{n\in\mathbb{N}}$ is a weak solution to \eqref{eq:SDE_Ito_cutoff}. Let us take a weakly convergent subsequence and still denote it by $\{(X^{n,M},W )\}_{n\in\mathbb{N}}$ for brevity. By the Skorokhod theorem (see e.g.~\cite[Theorem 2.4]{da2014stochastic}), there exist a complete probability space $(\widetilde\Omega,\widetilde{\mathcal{A}},\widetilde{\mathbb P})$ and $C\left([0,T];E_0\right)\times C\left([0,T];E_0\right)$-valued random variables $\{(\tilde X^{n,M},\tilde W^n)\}_{n\in\N}$, $(\tilde X^M, \tilde W)$ on $(\widetilde\Omega,\widetilde{\mathcal{A}},\widetilde{\mathbb P})$, such that $(\tilde X^{n,M},\tilde W^n)$ has the same law as $( X^{n,M},W^n)$ for every $n\in\mathbb{N}$, and the sequence $(\tilde X^{n,M},\tilde W^n)$ converges to $(\tilde X^M, \tilde W)$ in $C\left([0,T];E_0\right)\times C\left([0,T];E_0\right)$ $\widetilde \Prob$-a.s.. 

For technicalities, we denote the filtration generated by $\tilde X^M$, $\tilde W$ and $\widetilde \Prob$-null sets by $(\mathcal{\widetilde G}_t)_{t\in[0,T]}$, and we define $\mathcal{\widetilde F}_t \coloneqq \cap_{s>t} \mathcal{\widetilde G}_s$; clearly the filtration $\mathcal{\widetilde F}_t$ satisfies the standard assumption. The filtrations $(\mathcal{\widetilde G}^n_t)_{t\in[0,T]}$ and $(\mathcal{\widetilde F}^n_t)_{t\in[0,T]}$ for $\tilde X^{n,M},\tilde W^n$ are defined similarly. For notational convenience we dropped the dependence on $M$ in these filtrations (and in $\tilde W^n, \tilde W$). With regard to $\tilde W$, $\tilde W^n$, and $\{(\tilde X^{n,M},\tilde W^n)\}_{n\in\N}$, we have the next two technical and classical lemmas, we omit their proofs as they are completely analogous to those of e.g. \cite[Lemmas 4.1 and 4.2]{BMX2023}.

\begin{lemma}\label{BM}
For every $n\in\mathbb{N}_+$, $\tilde W^n$ is an $(\mathcal{\widetilde F}^n_t)_{t\in[0,T]}$-adapted $Q$-Wiener process, and the limit process $\tilde W$ is an $(\mathcal{\widetilde F}_t)_{t\in[0,T]}$-adapted $Q$-Wiener process. In particular, 
\begin{align*}
    &\tilde W^n=\sum_{k=1}^\infty \sqrt{\lambda_k}w_k \tilde W^n_k,
    &\tilde W=\sum_{k=1}^\infty \sqrt{\lambda_k}w_k \tilde W_k,
\end{align*}
where $\tilde W^n_k$ are $(\mathcal{\widetilde F}^n_t)_{t\in[0,T]}$-adapted mutually independent one-dimensional Brownian motions and $\tilde W_k$ are $(\mathcal{\widetilde F}_t)_{t\in[0,T]}$-adapted mutually independent one-dimensional Brownian motions. Moreover, for every $k\in\mathbb{N}_+$ $\tilde W^n_k \to \tilde W_k$ in $C([0,T],\R)$ $\tilde P$-a.s..
\end{lemma}

\begin{lemma}\label{WS}
For every $n\in \N$, $(\widetilde \Omega,  \widetilde{\mathcal{A}},(\widetilde{\mathcal{F}}^n_t)_{t}, \widetilde \Prob,\tilde X^{n,M}, \tilde W^n)$ is a weak solution of \eqref{eq:SDE_Galerkin}. 
\end{lemma}

Now we show that $(\tilde X^M,\tilde W)$, the $\Prob$-a.s.\ limit of $\{(\tilde X^{n,M},\tilde W^n)\}_{n\in\N}$, is a global weak solution of the SDE in It\^o form with cutoff \eqref{eq:SDE_Ito_cutoff}.

\begin{lemma}\label{lemma:exist_global_cutoff}
$(\widetilde \Omega,  \widetilde{\mathcal{A}},(\widetilde{\mathcal{F}}_t)_{t}, \widetilde \Prob,\tilde X^{M}, \tilde W)$ is a global-in-time, weak solution  to~\eqref{eq:SDE_Ito_cutoff} with respect to the initial condition $x_0 \in E_1$.
\end{lemma}

\begin{proof}
We have proved that on a common probability space $(\widetilde{\Omega}, \widetilde{\mathcal{A}}, \widetilde{\Prob})$ there exist random variables $\{(\tilde X^{n,M},\tilde W^n)\}_{n\in\N}$ and $(\tilde X^M, \tilde W)$ such that, $\widetilde{\Prob}$-a.s., $(\tilde X^{n,M},\tilde W^n)$ converges to $(\tilde X^M, \tilde W)$ in $C\left([0,T];E_0\right)\times C\left([0,T];E_0\right)$ and satisfies $\forall t\in[0,T]$
\begin{equation}\label{eq:tilde_X}
\begin{split}
        \tilde X^{n,M}_t - \tilde X^{n,M}_0 =&\int_0^t \theta_M(\lVert \tilde X^{n,M}_s \rVert_{E_0})  b^n(\tilde X^{n,M}_s)\,ds\\
        &\;+\frac12\int_0^t \theta_M(\lVert \tilde X^{n,M}_s \rVert_{E_0})\sum_{k=1}^n \left(\sigma\cdot\nabla\sigma\right)^n_k (\tilde X^{n,M}_s)\,ds\\
        &\;+\int_0^t \theta_M(\lVert \tilde X^{n,M}_s \rVert_{E_0})\sum_{k=1}^n \sigma^n_k (\tilde X^{n,M}_s) \,d  (\tilde W^n_k)_s\\
        &=:A+B+C.
\end{split}
\end{equation}
We want to pass to the limit in \eqref{eq:tilde_X} with respect to the $E_{-1}$ topology for $\tilde \Prob$-a.s.\ $\tilde \omega \in \tilde \Omega$. By the continuous embedding from $E_0$ to $E_{-1}$, $\Vert \tilde X^{n,M}_t-\tilde X_t\Vert_{E_{-1}}\to 0$ at all $t$ in $[0,T]$ $\widetilde{\Prob}$-a.s.\ as $n\to\infty$. Let us analyze the integral terms in~\eqref{eq:tilde_X}. Concerning the first deterministic integral $A$, we have
\begin{align*}
    &\left\Vert
    \int_0^t \theta_M(\lVert \tilde X^{n,M}_s \rVert_{E_0})\Pi_n b(\tilde X^{n,M}_s)\,ds
    -\int_0^t \theta_M(\lVert \tilde X^{M}_s \rVert_{E_0})b(\tilde X^M_s)\,ds
    \right\Vert_{E_{-1}}
  \\
  &\qquad\qquad \le 
    \int_0^t \theta_M(\lVert \tilde X^{n,M}_s \rVert_{E_0})\left\Vert \Pi_n b(\tilde X^{n,M}_s)-\Pi_n b(\tilde X^M_s) \right\Vert_{E_{-1}} ds\\
  &\qquad\qquad \quad+ \int_0^t \theta_M(\lVert \tilde X^{n,M}_s \rVert_{E_0}) \left\Vert \Pi_n b(\tilde X^M_s)-b(\tilde X^M_s) \right\Vert_{E_{-1}} ds\\
  &\qquad\qquad \quad+ \int_0^t \Bigl|\theta_M(\lVert \tilde X^{n,M}_s \rVert_{E_0}) -\theta_M(\lVert \tilde X^{M}_s \rVert_{E_0})\Bigr|\left\Vert b(\tilde X^M_s) \right\Vert_{E_{-1}} ds\\
  &\qquad\qquad \le 
    \int_0^t \left\Vert b(\tilde X^{n,M}_s)-b(\tilde X^M_s) \right\Vert_{E_{-1}} ds + \int_0^t \left\Vert \Pi_n b(\tilde X^M_s)-b(\tilde X^M_s) \right\Vert_{E_{-1}} ds\\
  &\qquad\qquad \quad+ \int_0^t \Bigl|\theta_M(\lVert \tilde X^{n,M}_s \rVert_{E_0}) -\theta_M(\lVert \tilde X^{M}_s \rVert_{E_0})\Bigr|\left\Vert b(\tilde X^M_s) \right\Vert_{E_{-1}} ds,
\end{align*}
where we exploited the projection property of $\Pi_n$ in $E_{-1}$. Then, using the continuity and boundedness on balls of $b$ (see Assumption~\ref{Assump:drift}), the continuity of the cutoff $\theta_M$, we may apply the dominated convergence theorem and conclude that the last integrals in the above inequality converge to zero $\widetilde{\Prob}$-a.s.. Thus we have for the first deterministic integral $A$ in~\eqref{eq:tilde_X}, as $n\to\infty$, 
\begin{equation}\label{eq:conv_A}
     \int_0^t \theta_M(\lVert \tilde X^{n,M}_s \rVert_{E_0}) b^n(\tilde X^{n,M}_s)\,ds
    \to \int_0^t \theta_M(\lVert \tilde X^M_s \rVert_{E_0})b(\tilde X^M_s)\,ds \quad \text{in } E_{-1}\quad\widetilde{\Prob}\text{-a.s..}
\end{equation}
Concerning the second deterministic integral $B$ in~\eqref{eq:tilde_X}, we have similarly
\begin{align*}
    &\left\Vert
    \int_0^t \theta_M(\lVert \tilde X^{n,M}_s \rVert_{E_0})\sum_{k=1}^n \left(\sigma\cdot\nabla\sigma\right)^n_k (\tilde X^{n,M}_s)\,ds
    -\int_0^t \theta_M(\lVert \tilde X^{M}_s \rVert_{E_0}) \left(\sigma\cdot\nabla\sigma\right) (\tilde X^{M}_s)\,ds
    \right\Vert_{E_{-1}}
  \\
  &\qquad\qquad \le 
    \int_0^t \theta_M(\lVert \tilde X^{n,M}_s \rVert_{E_0})\left\Vert \Pi_n \sum_{k=1}^n \left(\sigma\cdot\nabla\sigma\right)_k (\tilde X^{n,M}_s)-\Pi_n  \left(\sigma\cdot\nabla\sigma\right) (\tilde X^{M}_s) \right\Vert_{E_{-1}} ds\\
  &\qquad\qquad \quad+ \int_0^t \theta_M(\lVert \tilde X^{n,M}_s \rVert_{E_0}) \left\Vert \Pi_n  \left(\sigma\cdot\nabla\sigma\right) (\tilde X^{M}_s)-\left(\sigma\cdot\nabla\sigma\right) (\tilde X^{M}_s) \right\Vert_{E_{-1}} ds\\
  &\qquad\qquad \quad+ \int_0^t \Bigl|\theta_M(\lVert \tilde X^{n,M}_s \rVert_{E_0}) -\theta_M(\lVert \tilde X^{M}_s \rVert_{E_0})\Bigr|\left\Vert  \left(\sigma\cdot\nabla\sigma\right) (\tilde X^{M}_s) \right\Vert_{E_{-1}} ds\\
  &\qquad\qquad \le 
    \int_0^t \left\Vert \sum_{k=1}^n \left(\sigma\cdot\nabla\sigma\right)_k (\tilde X^{n,M}_s)-\sum_{k=1}^\infty  \left(\sigma\cdot\nabla\sigma\right)_k (\tilde X^{M}_s) \right\Vert_{E_{-1}} ds\\
    &\qquad\qquad \quad+\int_0^t \left\Vert \Pi_n  \left(\sigma\cdot\nabla\sigma\right) (\tilde X^{M}_s)-\left(\sigma\cdot\nabla\sigma\right) (\tilde X^{M}_s) \right\Vert_{E_{-1}} ds\\
  &\qquad\qquad \quad+ \int_0^t \Bigl|\theta_M(\lVert \tilde X^{n,M}_s \rVert_{E_0}) -\theta_M(\lVert \tilde X^{M}_s \rVert_{E_0})\Bigr|\left\Vert \left(\sigma\cdot\nabla\sigma\right) (\tilde X^{M}_s) \right\Vert_{E_{-1}} ds\\
  &\qquad\qquad =:B_1+B_2+B_3.
\end{align*}
We will show that all three terms in the right-hand side go to zero. Concerning the convergence of the first term $B_1$, by \eqref{eq:sigma_property_3} and dominated convergence theorem, for every $s \in [0,T]$, $\tilde{\mathbb{P}}$-a.s.,
\begin{align*}
    \sum_{k=1}^n \left(\sigma\cdot\nabla\sigma\right)_k (\tilde X^{n,M}_s)
    \to
    \sum_{k=1}^\infty  \left(\sigma\cdot\nabla\sigma\right)_k (\tilde X^{M}_s)
    \qquad\text{in $E_{-1}$},
\end{align*}
and, again by dominated convergence theorem, the term $B_1$ goes to zero $\tilde{\mathbb{P}}$-a.s..
The convergence of $B_2$ and $B_3$ follows in a similar way, exploiting the property of $\Pi_n$, the continuity of $\theta_M$ and \eqref{eq:sigma_property_4}.
Hence we have for the second deterministic integral term $B$ in \eqref{eq:tilde_X}, as $n\to\infty$, 
\begin{equation}\label{eq:conv_B}
     \int_0^t \theta_M(\lVert \tilde X^{n,M}_s \rVert_{E_0})\sum_{k=1}^n \left(\sigma\cdot\nabla\sigma\right)^n_k (\tilde X^{n,M}_s)\,ds
    \to \int_0^t \theta_M(\lVert \tilde X^{M}_s \rVert_{E_0}) \left(\sigma\cdot\nabla\sigma\right) (\tilde X^{M}_s)\,ds \quad \text{in } E_{-1}\quad\widetilde{\Prob}\text{-a.s..}
\end{equation}
Regarding the stochastic integral term $C$ in \eqref{eq:tilde_X}, we rely on \cite[Lemma 4.3]{BMX2023}. Up to extracting a subsequence, in order to prove the convergence of this term $C$ it is enough to show that
\begin{equation} \label{eq:condition_for_stoch_convergence}
\begin{split}
    &\int_0^t \sum_{k=1}^\infty \Bigl\Vert \mathds 1_{k\le n}\theta_M(\lVert \tilde X^{n,M}_s \rVert_{E_0}) \sigma^n_k (\tilde X^{n,M}_s)-\theta_M(\lVert \tilde X^{M}_s \rVert_{E_0}) \sigma_k (\tilde X^{M}_s)\Bigr\Vert^2_{E_{-1}}\,ds \to 0 \quad \widetilde{\Prob}\text{-a.s.}.
\end{split}
\end{equation}
We now show \eqref{eq:condition_for_stoch_convergence}. For every $k\in \N_+$ fixed, for every $s$, we have $\tilde{\mathbb{P}}$-a.s.\ as $n\to\infty$,
\begin{align*}
    \Bigl\Vert \mathds 1_{k\le n}\theta_M(\lVert \tilde X^{n,M}_s \rVert_{E_0}) \Pi_n \sigma_k (\tilde X^{n,M}_s)-\theta_M(\lVert \tilde X^{M}_s \rVert_{E_0}) \sigma_k (\tilde X^{M}_s)\Bigr\Vert_{E_{-1}} \to 0.
\end{align*}
Moreover, since 
\begin{align*}
    &\Bigl\Vert \mathds 1_{k\le n}\theta_M(\lVert \tilde X^{n,M}_s \rVert_{E_0}) \Pi_n \sigma_k (\tilde X^{n,M}_s)-\theta_M(\lVert \tilde X^{M}_s \rVert_{E_0}) \sigma_k (\tilde X^{M}_s)\Bigr\Vert^2_{E_{-1}}\\
    &\qquad\le 2\Bigl\Vert  \Pi_n \sigma_k (\tilde X^{n,M}_s)\Bigr\Vert^2_{E_{-1}} +  2\Bigl\Vert \sigma_k (\tilde X^{M}_s)\Bigr\Vert^2_{E_{-1}}
    \le 2\Bigl\Vert  \sigma_k (\tilde X^{n,M}_s)\Bigr\Vert^2_{E_{-1}} +  2\Bigl\Vert \sigma_k (\tilde X^{M}_s)\Bigr\Vert^2_{E_{-1}}
\end{align*}
the integrand in \eqref{eq:condition_for_stoch_convergence} is uniformly bounded thanks to \eqref{eq:sigma_property_1}. Therefore, by dominated convergence theorem, we obtain the convergence in \eqref{eq:condition_for_stoch_convergence} and thus, up to subsequence extraction, we have on the stochastic integral term $C$ \eqref{eq:tilde_X}, as $n\to\infty$,
\begin{equation}\label{eq:conv_C}
    \int_0^t \theta_M(\lVert \tilde X^{n,M}_s \rVert_{E_0})\sum_{k=1}^n\sigma^n_k (\tilde X^{n,M}_s)\,d  (\tilde W^n_k)_s
    \to \int_0^t \theta_M(\lVert \tilde X^M_s \rVert_{E_0})\sum_{k=1}^\infty \sigma_k (\tilde X^M_s)\,d  (\tilde W_k)_s \quad\text{in } E_{-1}\quad\widetilde{\Prob}\text{-a.s..}
\end{equation}

Putting together~\eqref{eq:conv_A},~\eqref{eq:conv_B} and~\eqref{eq:conv_C}, we can pass to the limit in~\eqref{eq:tilde_X} and get that $(\tilde X^M, \tilde W)$ is a global weak solution to~\eqref{eq:SDE_Ito_cutoff} with the initial condition $x_0$. The proof is complete.
\end{proof}

\subsection{Local maximal solutions} \label{subsecton:local_solution}

In this subsection, we glue together the solutions $X_M$ to \eqref{eq:SDE_Ito_cutoff}, corresponding to different values of $M$, building a local maximal solution $X$ to \eqref{eq:SDE_Ito}. This is a consequence of a uniqueness result for \eqref{eq:SDE_Ito_cutoff}. Indeed, thanks to uniqueness, we gain the probabilistically strong existence, ensuring that solutions to \eqref{eq:SDE_Ito_cutoff} with respect to different values of $M$ can be defined with respect to the same stochastic basis. Moreover, still due to uniqueness, the gluing procedure become unambiguous. 

\begin{lemma}\label{lemma:stopping_time_existence}
    Let $(\Omega,\mathcal{A},(\mathcal{F}_t)_t,\Prob)$ be a filtered probability space and $(\mathcal{F}_t)_t$ a right-continuous filtration. Let $Y$ be an adapted process with values in $E_0$ and trajectories in $L^\infty([0,T];E_1)\cap C([0,T];E_0)$, $\Prob$-a.s.. Then, there exists a sequence $(\tau_k)_{k\in \N}$ of stopping times with the following two properties:
    \begin{enumerate}
        \item $\lVert Y \rVert_{L^\infty([0,\tau_k];E_1)}\le k$ $\Prob$-a.s.,
        \item $\tau_k \uparrow T$ $\Prob$-a.s. as $k\to \infty$.
    \end{enumerate}
\end{lemma}
\begin{proof}
Introduce the exit times for the process $Y$ from the open ball $B_{k;E_1}$, which are defined as follow
\begin{align*}
    \tau_k\coloneqq \inf \{ t: \lVert Y_t\rVert_{E_1}\ge k\}.
\end{align*}
The fact that $\left(\tau_k\right)_{k\in\N}$ is a sequence of stopping times is a consequence of the lower semi-continuity of the norm $\lVert\cdot \rVert_{E_1}: E_0 \to [0,+\infty]$ (Remark \ref{remark:norm_semi-continuity}) and the continuity of the paths of $Y$ in $E_0$. In particular, we have to show that 
\begin{align*}
    \{\tau_k \le t\} =  \bigcap_{n\in\N} \left\{\sup_{s\in[0,t+\frac 1n]}\lVert Y_s\rVert_{E_1}\ge k\right\} \in \mathcal{F}_t. 
\end{align*}
Exploiting the right-continuity of the filtration, it is sufficient to show that, for all $n\in \N$,
\begin{align*}
    \left\{\sup_{s\in[0,t+\frac 1n]}\lVert Y_s\rVert_{E_1}\ge k\right\} \in \mathcal{F}_{t+\frac1n}.
\end{align*}
The previous relation holds, because, by lower semi-continuity, the function
\begin{align*}
   \sup_{s\in[0,t+\frac 1n]}\lVert Y_s\rVert_{E_1}=\sup_{s\in[0,t+\frac 1n] \cap \Q}\lVert Y_s\rVert_{E_1}
\end{align*}
is $\mathcal{F}_{t+\frac1n}$-measurable, being the supremum of a countable number of $\mathcal{F}_{t+\frac1n}$-measurable functions.
When comes to the two properties,
clearly the first property is satisfied simply by the definition of the exit times. Concerning the second property, it is sufficient to show that $\sup_{s\in [0,T]} \lVert Y_s \rVert_{E_1} = \lVert Y \rVert_{L^\infty([0,T],E_1)}$ $\Prob$-a.s.. The inequality $\sup_{s\in [0,T]} \lVert Y_s \rVert_{E_1} \ge \lVert Y \rVert_{L^\infty([0,T],E_1)}$ is trivial; we show the opposite inequality arguing by contradiction. Assume that there exists $\hat \omega$ such that $Y(\hat \omega) \in L^\infty([0,T];E_1)\cap C([0,T];E_0)$, while $\sup_{s\in [0,T]} \lVert Y_s (\hat \omega) \rVert_{E_1} > \lVert Y (\hat \omega)\rVert_{L^\infty([0,T],E_1)}$. Then, there exists $\hat s \in [0,T]$ such that $\lVert Y_{\hat s} \rVert_{E_1} > \lVert Y \rVert_{L^\infty([0,T],E_1)}$. On the other hand, by lower semi-continuity, $\forall \varepsilon >0$ there exists $\delta>0$ such that $\forall s \in B_\delta(\hat s) \cap [0,T]$ we have $\lVert Y_{s} \rVert_{E_1} > \lVert Y_{\hat s} \rVert_{E_1} -\varepsilon $, which implies that $ \lVert Y \rVert_{L^\infty([0,T],E_1)} \ge \lVert Y_{\hat s}  \rVert_{E_1} $, showing a contradiction.
\end{proof}

\begin{lemma} \label{lemma:uniqueness_cutoff}
Let $M\le M'$ and let $X^M$, $X^{M'}$ be two global solution to \eqref{eq:SDE_Ito_cutoff} relative to the same stochastic basis $( \Omega,  \mathcal{A},(\mathcal{F}_t)_{t},  \Prob,  W)$ and initial condition $x_0$, but cutoff functions parametrized by possibly different $M$ and $M'$ respectively. Let
\begin{align} \label{eq:stopping_times_for_X}
\begin{split}
    \tau_M=\inf\{t\mid \|X_t^M\|_{E_0}\ge M\},\\
    \tau'_M=\inf\{t\mid \|X_t^{M'}\|_{E_0}\ge M\}.
\end{split}
\end{align}
Then $\Prob$-a.s.\ $\tau_M=\tau'_M$ and $X_{\cdot\wedge \tau_M}^M=X_{\cdot\wedge \tau'_M}^{M'}$.
\end{lemma}

\begin{proof}[Proof of Lemma~\ref{lemma:uniqueness_cutoff}]
Let $\{\sigma^M_N\}_{N\in \N}$ and $\{\sigma^{M'}_N\}_{N\in \N}$ be stopping times relative to $X^M$ and $X^{M'}$ respectively as specified by Lemma \ref{lemma:stopping_time_existence}. Let $\sigma_N\coloneqq \sigma^M_N \wedge \sigma^{M'}_N \wedge \tau_M \wedge \tau'_M$.
Denote the difference process by $Y\coloneqq X^M-X^{M'}$. Its stochastic differential on $[0, \sigma_N)$ is
\begin{align*}
    dY_t=
    &\Big(\theta_M(\lVert X^M_t \rVert_{E_0})b(X^M_t)-\theta_M(\lVert X^{M'}_t \rVert_{E_0})b(X^{M'}_t) \Big) dt\\
    &+\frac12\Big(\theta_M(\lVert X^M_t \rVert_{E_0})\left(\sigma\cdot \nabla \sigma\right)(X^M_t)-\theta_M(\lVert X^{M'}_t \rVert_{E_0})\left(\sigma\cdot \nabla \sigma\right)(X^{M'}_t) \Big) dt\\
    &+\sum_{k\ge1}\Big(\theta_M(\lVert X^M_t \rVert_{E_0})\sigma_k(X^M_t)-\theta_M(\lVert X^{M'}_t \rVert_{E_0})\sigma_k(X^{M'}_t) \Big)dW^k_t\\
    &=\left(b(X^M_t)-b(X^{M'}_t) \right)dt  +\frac12\left(\left(\sigma\cdot \nabla \sigma\right)(X^M_t)-\left(\sigma\cdot \nabla \sigma\right)(X^{M'}_t) \right)dt\\
    &+\sum_{k\ge1}\left(\sigma_k(X^M_t)-\sigma_k(X^{M'}_t) \right)dW^k_t,
\end{align*}
and, by It\^o formula,
\begin{align*}
    d\lVert Y_t\rVert^2_{E_0}
    & =2\langle Y_t, b(X^M_t)-b(X^{M'}_t) \rangle_{E_0} dt\\
    &\quad+\langle Y_t, \left(\sigma\cdot \nabla \sigma\right)(X^M_t)-\left(\sigma\cdot \nabla \sigma\right)(X^{M'}_t) \rangle_{E_0}dt\\
    &\quad+\sum_{k\ge1} \left\Vert  \sigma_k(X^M_t)-\sigma_k(X^{M'}_t) \right\Vert^2_{E_0} dt\\
    &\quad+2\sum_{k\ge1}\langle Y_t,\sigma_k(X^M_t)-\sigma_k(X^{M'}_t) \rangle_{E_0} dW^k_t\\
    &=J_d dt + J_s dW_t.
\end{align*}
Let us analyse the above formula terms with the aim of applying a Gronwall estimate. By Assumption~\ref{Assump:drift}, estimates \eqref{eq:sigma_property_5} and \eqref{eq:sigma_property_2}, it follows that
\begin{align*}
    J_d 
    &\le 2 |\langle Y_t,b(X^M_t)- b(X^{M'}_t)\rangle_{E_0}| + \Vert Y_t \Vert_{E_0} \left\Vert \left(\sigma\cdot \nabla \sigma\right)(X^M_t)-\left(\sigma\cdot \nabla \sigma\right)(X^{M'}_t)  \right\Vert_{E_0}\\
    &\quad +  \sum_{k\ge1} \left\Vert  \sigma_k(X^M_t)-\sigma_k(X^{M'}_t) \right\Vert^2_{E_0} \\
    &\le 2 \chi (\Vert X^M_t\Vert_{E_1}\vee\Vert X^{M'}_t\Vert_{E_1} ) \lVert Y_t\rVert^2_{E_0} + 2h (\Vert X^M_t\Vert_{E_0}\vee\Vert X^{M'}_t\Vert_{E_0} )\Vert Y_t\Vert^2_{E_0}\\
    &\le C_N\Vert Y_t\Vert^2_{E_0}.
\end{align*}
We handle the stochastic term by Burkholder-Davis-Gundy inequality and estimate \eqref{eq:sigma_property_2}
\begin{align*}
\E \sup_{t\in [0,S]} &\left|\int_0^{t\wedge \sigma_N} \sum_{k\ge1}\langle Y_r,\sigma_k(X^M_r)-\sigma_k(X^{M'}_r) \rangle_{E_0} dW^k_r\right|\\
&\le C\E \left( \int_0^{S\wedge \sigma_N} \sum_{k\ge1} \left|\langle Y_r,\sigma_k(X^M_r)- \sigma_k(X^{M'}_r)\rangle_{E_0}\right|^2\,dr\right)^\frac12\\
&\le C\E \left( \int_0^{S\wedge \sigma_N} \Vert Y_r\Vert^2_{E_0}  \sum_{k\ge1} \Vert \sigma_k(X^M_r)- \sigma_k(X^{M'}_r)\Vert^2_{E_0}\,dr\right)^\frac12\\ 
&\le C_N \E \left( \int_0^{S\wedge \sigma_N} \lVert Y_r \rVert^4_{E_0}\,dr\right)^\frac12 \le C_N \E \left( \sup_{t\in[0,S\wedge \sigma_N]} \lVert Y_t \rVert^2_{E_0} \int_0^{S} \sup_{t\in[0,r\wedge \sigma_N]} \lVert Y_t \rVert^2_{E_0}\,dr\right)^\frac12
\end{align*}
and by Young inequality we end up with
\begin{align*}
    \E \sup_{t\in [0,S]} \left|\int_0^{t\wedge \sigma_N} J_s \,dW \right|\le \frac12 \E \sup_{t\in[0,S\wedge \sigma_N]} \lVert Y_t \rVert^2_{E_0} + C_N \int_0^{S} \E \sup_{t\in[0,r\wedge \sigma_N]} \lVert Y_t \rVert^2_{E_0}\,dr.
\end{align*}
Combining all the previous estimate, the following holds
\begin{align*}
    \E \sup_{t\in[0,S\wedge \sigma_N]} \lVert Y_t \rVert^2_{E_0} \le C_N \int_0^{S} \E \sup_{t\in[0,r\wedge \sigma_N]} \lVert Y_t \rVert^2_{E_0}\,dr \qquad \forall S>0,
\end{align*}
and thus, by Gronwall inequality, 
\begin{align*}
    \E \sup_{t\in[0,T\wedge \sigma_N]} \lVert Y_t \rVert^2_{E_0} =0.
\end{align*}
Since both $\sigma^M_N\to T$ and $\sigma^{M'}_N\to T$ $\Prob$-a.s.,  as $N\to\infty$, then $Y_t=0$ for all $t\in [0,\tau_M \wedge \tau'_M]$ $\Prob$-a.s.. As a consequence, we also have $\tau_M = \tau'_M$ $\Prob$-a.s..
\end{proof}

Thanks to Lemma \ref{lemma:uniqueness_cutoff} we can unambiguously define the local maximal solution to \eqref{eq:SDE_Ito}. First, let $\tau_M$ be as in \eqref{eq:stopping_times_for_X} and introduce $\tau\coloneqq \sup_{M\in\mathbb{N}}\tau_M$. We define the process $(X_t)_{t\in[0,\tau)}$ by the following formula
\begin{align} \label{eq:glued_solution}
    X_t = X_t^M \text{ for }t<\tau_M.
\end{align}
Note that $\tau_M\uparrow \tau$, $\tau_M<\tau$ for every $M$, hence $\tau$ is accessible; in particular $\tau>0$ $\Prob$-a.s..

\begin{lemma} \label{lemma:properties_glued_solution}
    The process $X$ defined by \eqref{eq:glued_solution} satisfies the following properties:
    \begin{enumerate}
        \item $X\in C([0,\tau);E_0)$ $\Prob$-a.s.;
        \item $X\in L^\infty([0,\tau_M];E_1)$ $\Prob$-a.s. for every $M$;
        \item $X$ is a local solution to \eqref{eq:SDE_Ito} on $[0,\tau)$;
        \item $\sup_{t<\tau}\|X_t\|_{E_0}=+\infty$ on $\{\tau<\infty\}$.
    \end{enumerate}
\end{lemma}
\begin{proof}
    The first two properties follow from the analogous properties satisfied by the processes $X^M$. Concerning the third property, it is enough to note that, since $X=X^M$ on $[0,\tau_M)$, then
    \begin{align*}
        \int_0^t \sigma(X_r)dW_r = \int_0^t \sigma(X^M_r)dW_r,\quad \forall t\in [0,\tau_M],
    \end{align*}
    therefore \eqref{eq:SDE_Ito} is satisfied on $[0,\tau_M]$ for every $M$, and so on $[0,\tau)$.
    Concerning the last property, if $\tau<\infty$, then $\tau_M<\infty$ for every $M$ and so
    \begin{align*}
        \sup_{t<\tau}\|X_t\|_{E_0}\ge \sup_M \|X_{\tau_M}\|_{E_0} = \sup_M M =\infty.
    \end{align*}
    The proof is complete.
\end{proof}

\section{Global well-posedness by Lyapunov function method} \label{section:global_solution}

In this section, we exhibit a Lyapunov function for the SDE \eqref{eq:SDE_Ito} to show the global well-posedness.

Let $R$ be the radius introduced in Assumption~\ref{Assump:drift}-(b), and let $a\in (0,\log(2R))$. We introduce a radially symmetric $C^2$ Lyapunov function $V:E_0 \to \R$, non-decreasing in the $E_0$ norm, such that
\begin{align}\label{eq:Lyapunov_function}
\begin{cases}
    V(x)\ge a, &\forall x\in E_0, \\
    V(x)= a, &\forall x\in B_{R;E_0}, \\
    V(x)=\log\lVert x \rVert_{E_0}, \quad &\forall x\in B^c_{2R;E_0}.
\end{cases}
\end{align}
Explicit computations show that for all $x\in B^c_{2R;E_0}$, with respect to the $E_0$ scalar product,
\begin{align}\label{eq:Lyapunov_derivatives}
    \nabla V(x)&=\frac x{\lVert x \rVert_{E_0}^2},\\
    D^2V(x)&=\frac 1{\lVert x \rVert_{E_0}^2}\left(\Id -2\frac{x\otimes x}{\lVert x \rVert_{E_0}^2} \right).
\end{align}
Thanks to this Lyapunov function, in the following theorem we can show that local solutions to \eqref{eq:SDE_Ito} are actually global.
\begin{theorem}[Global Existence] \label{thm:global_existence}
    Under Assumptions \ref{Assump:WienerProcess}, \ref{Assump:Spaces}, \ref{Assump:Projection}, \ref{Assump:drift} and \ref{Assump:sigma}, let in \eqref{eq:sigma} $\eta> m/2$, $c_N>0$ (or alternatively $\eta =m/2$, $c_N$ large enough) and $\alpha$ such that
    \begin{align}\label{eq:alpha_condition}
        1 < \alpha < 1+ \frac{1}{\eta-1} \frac{\Tr Q - \lVert Q \rVert}{\lVert Q \rVert}.
    \end{align}
    Then the stochastic process $(X_t)_{t\in [0,\tau)}$, defined in \eqref{eq:glued_solution}, is a global strong solution of \eqref{eq:SDE_ItoForm}, i.e. $\tau=\infty$ $\Prob$-a.s..
\end{theorem}
\begin{proof}
    To show that $\tau=\infty$ $\Prob$-a.s., we make use of the Lyapunov function $V$ introduced in \eqref{eq:Lyapunov_function}. By the It\^o formula, $V(X)$ has the following stochastic differential for $t\in [0,\tau)$,
\begin{align}\label{Ito-LV}
    dV(X_t)= \mathcal{L} V(X_t) dt + \sum_{k\ge 1}\langle \nabla V(X_t), \sigma_k(X_t)\rangle_{E_0} dW^k_t.
\end{align}
Above, $\mathcal{L}$ denotes the generator of this SDE, i.e.
\begin{align*}
    \mathcal{L}\left(V\right)(x)
    &= \langle \nabla V(x), b(x) \rangle_{E_0} + \frac12 \langle \nabla V(x), \left( \sigma\cdot \nabla \sigma\right)(x) \rangle_{E_0}+ \frac12\sum_{k\ge1} \langle \sigma_k(x), D^2V(x)\sigma_k(x) \rangle_{E_0}
\end{align*}
We show that $\mathcal{L}(V)$ is bounded in $E_0$. 
First, note
that $V$ is constant in $B_{R;E_0}$ and so $\mathcal{L}(V)(x)= 0$ for all $x\in B_{R;E_0}$. Next, we consider the region $R\le \lVert x \rVert _{E_0} \le 2R$. By the radial symmetry and the smoothness of $V$, $\nabla V(x)=\lVert \nabla V(x) \rVert_{E_0} \cdot x/\lVert x \rVert_{E_0}$ and both $\lVert \nabla V(x) \rVert_{E_0}$ and $\lVert D^2V(x)\rVert_{L(E_0,E_0)}$ are bounded on $B_{2R;E_0}$.  
Therefore, for $x \in E_0$ with $R\le \lVert x \rVert _{E_0} \le 2R$, exploiting Assumption \ref{Assump:drift}, \eqref{eq:sigma_property_1} and \eqref{eq:sigma_property_4}, we have

\begin{align*}
    \mathcal{L}\left(V\right)(x)
    &= \frac{\lVert \nabla V(x) \rVert_{E_0}}{\lVert x \rVert_{E_0}}\langle  x, b(x) \rangle_{E_0} + \frac12 \langle \nabla V(x), \left( \sigma\cdot \nabla \sigma\right)(x) \rangle_{E_0}\\
    & \quad+ \frac12\sum_{k\ge1} \langle \sigma_k(x), D^2V(x)\sigma_k(x) \rangle_{E_0}\\
    &\le g(\lVert x \rVert_{E_0})\lVert \nabla V(x) \rVert_{E_0} \lVert x \rVert_{E_0} + \frac12 \lVert \nabla V(x) \rVert_{E_0} \lVert\left( \sigma\cdot\nabla\sigma\right)(x) \rVert_{E_0}\\
    & \quad+ \frac12 \lVert D^2V(x)\rVert_{L(E_0,E_0)} \sum_{k\ge1} \lVert\sigma_k(x) \rVert^2_{E_0}\\
    &\le C_R,
\end{align*}
for some $C_R>0$.

Lastly, we consider the most important case $x\in B^c_{2R;E_0}$. In this region $\sigma$ takes the prescribed form given in Assumption \ref{Assump:sigma}-(c). In particular, recalling \eqref{eq:alpha_beta_gamma_delta},

\begin{align*}
    \mathcal{L}\left(V\right)(x)
    &= \langle \nabla V(x), b(x) \rangle_{E_0} + \frac12\sum_{k\ge1} \alpha_k(x)\langle \nabla V(x), w_k \rangle_{E_0} + \frac12\sum_{k\ge1} \beta_k(x)\langle \nabla V(x), x \rangle_{E_0}\\
    &\quad + \frac12\sum_{k\ge1} \langle \sigma_k(x), D^2V(x)\sigma_k(x) \rangle_{E_0}\\
    &\eqqcolon L_1 + L_2 +L_3 + L_4.
\end{align*}
Due to Assumption~\ref{Assump:drift} and \eqref{eq:Lyapunov_derivatives} we obtain
\begin{align*}
    L_1
    &= \langle \nabla V(x), b(x) \rangle_{E_0} = \frac{1}{\lVert x \rVert^2_{E_0}} \langle x, b(x) \rangle_{E_0} \le g(\lVert x \rVert_{E_0}),\\
    L_2
    &=\frac12\sum_{k\ge1} \alpha_k(x)\langle \nabla V(x), w_k \rangle_{E_0}\\
    &= c_N^2\frac12 (\eta - \alpha - \eta\alpha)\lVert x \rVert^{2\eta-4}_{E_0} \sum_{k\ge1} \lambda_k  \langle x,w_k\rangle^2_{E_0}, \\
    L_3
    &=\frac12\sum_{k\ge1} \beta_k(x)\langle \nabla V(x), x \rangle_{E_0}\\
    &= \frac12 c_N^2(\alpha^2\eta-\alpha(\eta-2)) \lVert x \rVert^{2\eta-4}_{E_0} \sum_{k\ge1}\lambda_k \langle x,w_k\rangle^2_{E_0} - \frac12 c_N^2 \alpha \lVert x \rVert^{2\eta-2}_{E_0}\sum_{k\ge1}\lambda_k \lVert w_k \rVert_{E_0}^2,\\
    L_4
    &=\frac12\sum_{k\ge1} \langle \sigma_k(x), D^2V(x)\sigma_k(x) \rangle_{E_0}\\
    &=\frac12 c_N^2(-\alpha^2+2\alpha-2) \lVert x \rVert^{2\eta-4}_{E_0} \sum_{k\ge1}\lambda_k \langle x,w_k\rangle^2_{E_0} + \frac12  c_N^2 \lVert x \rVert^{2\eta-2}_{E_0}\sum_{k\ge1}\lambda_k \lVert w_k \rVert_{E_0}^2.
\end{align*}
Let us introduce $\hat \alpha \coloneqq 1-\alpha$ and $\hat \beta \coloneqq \hat \alpha^2(\eta -1)- \hat\alpha$. Being $\hat \alpha<0$ and $\eta>1$, it follows $\hat \beta>0$. Then, if we combine all the terms, we obtain
\begin{align*}
    \mathcal{L}\left(V\right)(x) 
    &\le g(\|x\|_{E_0})+ \frac12 c_N^2\hat \alpha \lVert x \rVert^{2\eta-2}_{E_0}\sum_{k\ge1}\lambda_k \lVert w_k \rVert_{E_0}^2 + \frac12 c_N^2 \hat \beta \lVert x \rVert^{2\eta-4}_{E_0} \sum_{k\ge1}\lambda_k \langle x,w_k\rangle^2_{E_0}\\
    &= g(\|x\|_{E_0})+ \frac12 c_N^2 \hat \alpha \lVert x \rVert^{2\eta-2}_{E_0} \Tr Q+ \frac12 c_N^2 \hat \beta \lVert x \rVert^{2\eta-2}_{E_0} \frac{\lVert Q^\frac12 x \rVert_{E_0}^2}{\lVert x \rVert_{E_0}^2}\\
    &\le g(\|x\|_{E_0})+ \frac12 c_N^2 \hat \alpha \lVert x \rVert^{2\eta-2}_{E_0} \Tr Q+ \frac12 c_N^2 \hat \beta \lVert x \rVert^{2\eta-2}_{E_0} \lVert Q \rVert.
\end{align*}
With simple computations one can see that
\begin{align*}
     \hat \alpha \Tr Q+ \hat \beta  \lVert Q \rVert<0 \iff \alpha -1<  \frac{1}{\eta-1} \frac{\Tr Q - \lVert Q \rVert}{\lVert Q \rVert},
\end{align*}
Thus, under condition \eqref{eq:alpha_condition}, we have
\begin{align*}
    \delta\coloneqq -\hat \alpha \Tr Q- \hat \beta  \lVert Q \rVert>0.
\end{align*}
Therefore we get, for $x\in B^c_{2R;E_0}$,
\begin{align*}
    \mathcal{L}\left(V\right)(x) \le g(\|x\|_{E_0})-\frac\delta2 c_N^2  \lVert x \rVert^{2\eta-2}_{E_0}.
\end{align*}
Hence, by Assumption \ref{Assump:drift}-(b), when $\eta> m/2$ we obtain, for $x\in B^c_{2R;E_0}$,
\begin{align*}
\mathcal{L}\left(V\right)(x) \le C
\end{align*}
for some constant $C$ (independent of $x$). When $\eta = m/2$, the same conclusion holds if $\frac\delta2 c_N^2>c_D$.\\
Combining all cases, we conclude that $\mathcal{L}(V)$ is bounded on the whole space $E_0$. Since $0<a\leq V$, i.e.\ $V$ is bounded from below by a positive constant, while $\mathcal{L}(V)$ is bounded from above, there must be a positive constant $c$ such that 
\begin{equation}
\mathcal{L}\left(V\right)\le c V.\label{eq:Lyap_bd}
\end{equation}

Let $N> 2R$ and let $\tau_N=\inf\{t\ge 0\mid \|X_t\|_{E_0}\ge N\}$ be the first exit time from $B_{R;E_0}$, that is the first time when $V(X_t)$ hits $\log N$. By definition, $\tau_N\to \tau$ $\Prob$-a.s..
We integrate both sides of \eqref{Ito-LV} over the interval $[0, t \land \tau_N]$, take the expected value, and apply \eqref{eq:Lyap_bd}. We obtain
\begin{align*}    &\E[V(X_{t\land\tau_N})] =\E[V(X_0)]
    +\E\int_0^{t\land \tau_N}\mathcal{L}(V)(X_s)\,ds
    \\ & \qquad
    \le \E[V(X_0)]+\E\int_0^{t\land \tau_N}cV(X_s)\,ds, \quad t\in(0,T).
\end{align*}
By Gronwall's lemma, we get, for every $T>0$,
\begin{equation*}
\E[V(X_{T\land\tau_N})]\le\E[V(X_0)]e^{cT}=V(x_0)e^{cT},
\end{equation*}
which by the Markov inequality implies
\begin{equation} 
    \Prob\left(\sup_{t\in[0,T\wedge \tau)} \|X_t\|_{E_0} \ge N\right)=\Prob\left( \|X_{T\wedge \tau_N}\|_{E_0} \ge N\right)=\Prob\left(V(X_{T\wedge \tau_N})\ge \log N\right)
    \le\frac{\E[V(x_0)] e^{cT}}{\log N}.\label{eq:Vp_Markov}
\end{equation}
Therefore, letting $N \to\infty$, we obtain that $\sup_{t\in [0,T\wedge \tau)}\|X_t\|_{E_0}$ is finite $\Prob$-a.s., showing (by Lemma \ref{lemma:properties_glued_solution}) that $\tau\ge T$ $\Prob$-a.s.. By arbitrariness of $T$, $\tau=\infty$ $\Prob$-a.s.. The proof is complete.
\end{proof}

\begin{theorem}[Pathwise Uniqueness]\label{thm:pathwise_uniqueness}
Let $(X_t)_{t\in [0,\sigma)}$ and $(Y_t)_{t\in [0,\tau)}$ be two solutions of \eqref{eq:SDE_Ito}, relative to the same stochastic basis $( \Omega,  \mathcal{A},(\mathcal{F}_t)_{t},  \Prob,  W)$ and the same initial condition $x_0$. Then $\Prob$-a.s.\  $X_{\cdot\wedge\sigma\wedge\tau}=Y_{\cdot\wedge\sigma\wedge\tau}$.
\end{theorem}

\begin{proof}[Proof of Theorem~\ref{thm:pathwise_uniqueness}]
The proof is an easy adaptation of the proof of Lemma \ref{lemma:uniqueness_cutoff}.
\end{proof}

\section{The stochastic Euler equations} \label{section:Euler}

In this section, we address the stochastic Euler equation on a $d$-dimensional smooth (i.e.\ of class $C^\infty$), bounded, and simply connected domain $D$, with $d=2,3$:
\begin{align}\label{eq:stoch_Euler}
\begin{cases}
    du + (u\cdot \nabla) u\, dt + \nabla p \,dt = \sigma(u) \circ dW_t &(t,x)\in [0,T]\times D\\
    \operatorname{div} u=0 &(t,x)\in [0,T]\times D\\
    u\cdot n=0 &(t,x)\in [0,T]\times \partial D\\
    u(0,x)=u_0(x) & x \in D
\end{cases}  
\end{align}
Here, $u:[0,T]\times D\times \Omega\to \R^d$ is the velocity field, $p:[0,T]\times D\times \Omega\to \R$ is the pressure field and $W$ is $Q$-Wiener process on a filtered probability space $(\Omega,\mathcal{A},(\mathcal{F}_t)_t,\Prob)$ satisfying the standard assumption. Moreover, $u_0$ is the deterministic initial datum, $n$ is the unit outer normal on $\partial D$, and the equation $u\cdot n=0$ encodes the slip boundary condition. Both $u$ and $p$ are unknown functions, whereas the noise coefficient $\sigma$ will be specified later. 

Let $s\in\mathbb{N}$, we define the following (standard) function spaces 
\begin{equation*}
\mathcal{H}^s(D)=\left\{u\in H^s(D)\mid \operatorname{div} u=0\text{ on }D\text{, }u\cdot n=0\text{ on }\partial D\right\},
\end{equation*}
where $H^s(D)$ refers to the set of functions whose weak derivatives up to order-$s$ belong to $L^2(D)$. We endow $\mathcal{H}^s(D)$ with the $H^s(D)$-scalar product. Note that $\mathcal{H}^s(D)$ is a separable Hilbert space for every positive integer $s$, as it is a closed subspace of the separable Hilbert space $H^s(D)$. With the notation, we now better specify $W$ and $\sigma$.

\begin{assumption}\label{hp:WienerProcess_Euler}
$W$ is a $Q$-Wiener process with values in $\mathcal{H}^s(D)$ for a trace-class operator $Q:\mathcal{H}^s(D)\to \mathcal{H}^s(D)$. Let $W=\sum_{k\ge1} \sqrt{\lambda_k}w_k W^k$, where $\{w_k\}_{k\in \N}$ is an orthonormal basis of $\mathcal{H}^s(D)$ made of eigenvectors of $Q$, $\{\lambda_k\}_{k\in\N}$ is the associated sequence of non-negative eigenvalues and $\{W^k\}_{k\in\N}$ are independent one dimensional Brownian motions defined on the filtered probability space $(\Omega,\mathcal{A},(\mathcal{F}_t)_t,\mathbb P)$.
We require 
\begin{equation*}
    \sum_{k=1}^\infty \lambda_k \lVert w_k\rVert^2_{\mathcal{H}^{s+1}(D)}<\infty.
\end{equation*}
\end{assumption}
\begin{assumption}\label{hp:noise_Euler}
The diffusion coefficient $\sigma:\mathcal{H}^s(D)\to \mathcal{H}^{s}(D)$ has the form
\begin{align} \label{eq:sigma_Euler}
\sigma(u)=c_N \psi (\lVert u \rVert_{H^s})\lVert u \rVert_{H^s}^\eta \left(I- \alpha \frac{u\otimes u}{\lVert u \rVert^2_{H^s}}\right)
\end{align}
for $\eta> 3/2$, $c_N>0$ (or alternatively $\eta =3/2$, $c_N$ large enough) and $\alpha$ satisfies
\begin{align}\label{eq:alpha_condition_Euler}
    1 < \alpha < 1+ \frac{1}{\eta-1} \frac{\Tr Q - \lVert Q \rVert}{\lVert Q \rVert}.
\end{align}
Moreover, $\psi :\R \to [0,1]$ is a smooth even function such that, for some $R>0$, $\psi(s)\equiv 0$ for $|s|\le R/2$ and $\psi(s)\equiv 1$ for $|s|\ge R$.
\end{assumption}
Recall the Leray projector $\mathcal{P}$, which is the orthogonal projection of $L^2(D)$ onto its closed subspace $\mathcal{H}^0(D)$. Equivalently, for every $v\in L^2(D)$, we have $\mathcal{P}v=v-\mathcal{Q}v$, where $\mathcal{Q}v=-\nabla w$, and $w$ is uniquely defined up to an additive constant by the elliptic Neumann problem
\begin{align*}
\begin{cases}
  -\Delta w = \operatorname{div} v & x\in D,\\
  \;\;\partial_n w =v\cdot n &x\in \partial D.
\end{cases}
\end{align*} 
Applying $\mathcal{P}$ to the first equation in \eqref{eq:stoch_Euler}, we can eliminate the pressure gradient and interpret the stochastic Euler equations as a stochastic evolution equation of type \eqref{eq:SDE_Strat_compact} on $\mathcal{H}^s(D)$ ($s\in \mathbb{N}$), for $b(u) = -\mathcal{P}[(u\cdot\nabla) u]$ and $\sigma(u)$ as in Assumption~\ref{hp:noise_Euler},
\begin{align}\label{eq:stoch_Euler_Cauchy}
\begin{cases}
    du = - \mathcal{P}(u\cdot \nabla) u\, dt + \sigma(u)\circ dW_t &(t,x)\in [0,T]\times D,\\
    u(0,x)=u_0(x) & x \in D.
\end{cases}    
\end{align}
Note that applying $\mathcal{P}$ is not restrictive, since one can retrieve the pressure $p$ a posteriori, starting from a solution $u$ to \eqref{eq:stoch_Euler_Cauchy}, by solving an elliptic problem.

The main result of this section is the following.
\begin{theorem}\label{thm:ApplicationToEuler}
Let $s\in \mathbb{N}$ and $s>d/2+1$. Under the Assumptions \ref{hp:WienerProcess_Euler} and \ref{hp:noise_Euler}, for every deterministic initial condition $u_0 \in \mathcal{H}^{s+1}(D)$, the stochastic Euler equation~\eqref{eq:stoch_Euler_Cauchy} has a unique, global-in-time, strong solution in $L^\infty([0,T];\mathcal{H}^{s+1}(D))\cap C([0,T];\mathcal{H}^{s}(D))$, $\Prob$-a.s., for all $T>0$.
\end{theorem}

To prepare for the proof, we first state a few (classical and standard) results.

\begin{lemma}\label{Euler:embed}
 Let $s\in \mathbb{N}$ and $s>1$. Then, $\mathcal{H}^{s+1}(D)$ is compactly embedded in $\mathcal{H}^{s}(D)$, and  $\mathcal{H}^{s}(D)$ is continuously embedded in $\mathcal{H}^{s-1}(D)$. Both embeddings are dense. Moreover, there exist $\theta\in (0,1)$ and $C>0$ such that 
 \begin{equation*}
  \|v\|_{H^{s}(D)}\le C\|v\|_{H^{s+1}(D)}^{1-\theta}\|v\|_{H^{s-1}(D)}^\theta  , \quad \forall v\in \mathcal{H}^{s+1}(D).
 \end{equation*}
\end{lemma}
\begin{proof}
See \cite[Lemma 5.4]{BMX2023}.  
\end{proof}

\begin{lemma}\label{Euler:basis}
 Let $s\in \mathbb{N}$ and $s>1$. Then, $\mathcal{H}^{s-1}(D)$ has an orthonormal basis $\{e_k\}_{k\in \N_+}$ consisting of elements in $\mathcal{H}^{s+2}(D)$. Moreover, $\{e_k\}_{k\in \N_+}$ is also orthogonal in $\mathcal{H}^{s+1}(D)$.
\end{lemma}
\begin{proof}
    See \cite[Lemma 5.5 and Remark 5.6]{BMX2023}.
\end{proof}
Next, we recall a few classical estimates (see, for instance, \cite{GlaVic2014}). Let $m$ be an integer satisfying $m>d/2$. Then,
\begin{itemize}
    \item (The Moser Estimate)~$ \|u v\|_{H^{m}} \le C (\|u\|_{L^\infty}\|v\|_{H^{m}} +\|u\|_{H^{m}}\|v\|_{L^{\infty}})$,
    \item (The Sobolev Embedding)~$\|u\|_{L^\infty}\leq C\|u\|_{H^{m}}$ and $\|u\|_{W^{1,\infty}}\leq C\|u\|_{H^{m+1}}$.
\end{itemize}
Hence, for all integers $m$ such that $m>d/2$, we infer 
\begin{align}
    &\|(u\cdot\nabla)v\|_{H^{m}} \le C (\|u\|_{L^\infty}\|v\|_{H^{m+1}} +\|u\|_{H^{m}}\|v\|_{W^{1,\infty}})\le C \|u\|_{H^{m+1}}\|v\|_{H^{m+1}}\label{eq:bd_drift_1} 
\end{align}
whenever $u,v\in \mathcal{H}^{m+1}(D)$. If $m>d/2+1$, then by~\cite[Lemma 2.1 (c)]{GlaVic2014},
\begin{align}
  \langle \mathcal{P}[(u\cdot \nabla)v], v\rangle_{H^m} \le C(\|u\|_{W^{1,\infty}}\|v\|_{H^{m}} +\|u\|_{H^{m}}\|v\|_{W^{1,\infty}})\|v\|_{H^m},\label{eq:bd_drift_2_2}
\end{align}
when $u\in \mathcal{H}^{m}(D)$ and $v\in \mathcal{H}^{m+1}(D)$. Note that \eqref{eq:bd_drift_2_2} is a result of the commutator estimates.

\begin{proof}[Proof of Theorem~\ref{thm:ApplicationToEuler}]
The theorem is an application of the main result, i.e.\ Theorem \ref{thm:main}, with $E_1=\mathcal H^{s+1}(D)$, $E_0=\mathcal H^{s}(D)$ and $E_{-1}=\mathcal H^{s-1}(D)$. We need to check Assumptions \ref{Assump:WienerProcess}, \ref{Assump:Spaces}, \ref{Assump:Projection}, \ref{Assump:drift} and \ref{Assump:sigma}. However, thanks to Assumption \ref{hp:WienerProcess_Euler}, Lemma \ref{Euler:embed}, Lemma \ref{Euler:basis} and Assumption \ref{hp:noise_Euler} (paired with Lemma \ref{lemma:coherence}), we are only left to check the assumptions on the drift coefficient, i.e.\ Assumption \ref{Assump:drift}.

Assumption~\ref{Assump:drift}-(a) follows from the continuity of the Leray projector, the estimate~\eqref{eq:bd_drift_1} and a simple triangle inequality (provided $s>d/2+1$). We move to Assumption \ref{Assump:drift}-(b). Exploiting the orthogonality of $\{e_k\}_{k\in \N_+}$ in both $\langle \cdot,\cdot\rangle_{H^{s-1}}$ and $\langle \cdot,\cdot\rangle_{H^{s+1}}$-scalar products and applying~\eqref{eq:bd_drift_2_2} with $m=s+1$, we obtain for every $u\in E^n$ (recall that $u\in E^n$ implies $u\in \mathcal{H}^{s+2}$, see Lemma \ref{Euler:basis}) that
\begin{align*}
    \langle b^n(u), u \rangle_{H^{s+1}}
    = \langle \Pi_n \mathcal{P} (u\cdot\nabla u) , u \rangle_{H^{s+1}}
    = \langle \mathcal{P} (u\cdot\nabla u), u \rangle_{H^{s+1}}
    \le  C \lVert u \rVert_{W^{1,\infty}} \lVert u \rVert_{H^{s+1}}^2,
\end{align*}
and similarly, for $u\in \mathcal{H}^{s+1}$,
\begin{align*}
    \langle b(u), u \rangle_{H^{s}}
    = \langle \mathcal{P} (u\cdot\nabla u), u \rangle_{H^{s}}
    \le  C \lVert u \rVert_{W^{1,\infty}} \lVert u \rVert_{H^{s}}^2.
\end{align*}
Hence, exploiting the Sobolev embedding of $H^{s}(D)$ into $W^{1,\infty}(D)$, Assumption~\ref{Assump:drift}-(b) is met with
\begin{equation} \label{eq:g(x)}
    g(\lVert u \rVert_{H^{s}})= c_D\lVert u \rVert_{H^{s}},
\end{equation}
for some $c_D>0$, so that we can choose $m=3$.
Assumption \ref{Assump:drift}-(c) holds since $b^n$ is quadratic. Finally, utilizing \eqref{eq:bd_drift_1}, \eqref{eq:bd_drift_2_2}, the properties of $\mathcal{P}$ and the Sobolev embeddings, we have
\begin{align*}
   & \langle \mathcal{P}[v\cdot \nabla v-w\cdot \nabla w],v-w\rangle_{H^{s}}
    = \langle \mathcal{P}[v\cdot \nabla(v-w)],v-w\rangle_{H^{s}} 
     + \langle \mathcal{P}[(v-w)\cdot \nabla w],v-w\rangle_{H^{s}}
     \\ & \qquad\qquad\qquad  
     \leq C
     (\|v\|_{W^{1,\infty}}\|v-w\|_{H^{s}} +\|v\|_{H^{s}}\|v-w\|_{W^{1,\infty}})\|v-w\|_{H^{s}}
      \\ & \qquad\qquad\qquad\quad  
      +C
     (\|v-w\|_{L^{\infty}}\|w\|_{H^{s+1}} +\|v-w\|_{H^{s}}\|w\|_{W^{1,\infty}})\|v-w\|_{H^{s}}
     \\ & \qquad\qquad \qquad  
     \leq
      C (\lVert v\rVert_{H^{s}} +\lVert w\rVert_{H^{s+1}} )\lVert v-w\rVert_{H^{s}}^2,
\end{align*}
which fulfils Assumption \ref{Assump:drift}-(d) if $s>d/2+1$.
\end{proof}

\appendix

\section{Proofs of technical lemmas} \label{section:appendix_proofs}

\subsection{Proof of Lemma \ref{lemma:noise_properties}} \label{subsection:proof_lemma_noise_properties}
\begin{proof}
    First, note that, for every $k\ge 1$,
    \begin{align*}
        \| \left(\sigma\cdot\nabla \sigma\right)_k(x)\|_{E_0}
        &=\lambda_k \|[\sigma(x)[w_k]\cdot\nabla (\sigma(x)[w_k])\|_{E_0}\\
        &\le \lambda_k \|\sigma(x)\|_{L(E_0,E_0)} \|D\sigma(x)\|_{L(E_0,L(E_0,E_0))},
    \end{align*}
    and, by the mean value theorem,
    \begin{align*}
        \Vert   \sigma_k(x)-\sigma_k(y)\Vert_{E_0}
        &=\sqrt{\lambda_k}\Vert  (\sigma(x)-\sigma(y))[w_k]\Vert_{E_0}\\
        &\le \sqrt{\lambda_k}\Vert  \sigma(x)-\sigma(y)\Vert_{L(E_0,E_0)} \\
        &\le  \sqrt{\lambda_k }\Vert   x-y\Vert_{E_0}\sup_{z\in B_{\lVert x \rVert \vee \lVert y \rVert;E_0}}\Vert  D\sigma(z)\Vert_{L(E_0,L(E_0,E_0))},
    \end{align*}
    \begin{align}\label{eq:estimate_sigma_nabla_sigma}
        \begin{split}
        \|\left(\sigma\cdot\nabla \sigma\right)_k(x)-\left(\sigma\cdot\nabla \sigma\right)_k(y)\|_{E_0}
        &=\lambda_k \|\sigma(x)[w_k]\cdot\nabla (\sigma(x)[w_k])-\sigma(y)[w_k]\cdot\nabla (\sigma(y)[w_k])\|_{E_0}\\
        &\le\lambda_k \|(\sigma(x)-\sigma(y))[w_k]\cdot\nabla (\sigma(x)[w_k])\|_{E_0}\\
        &\qquad+\lambda_k \|\sigma(y)[w_k]\cdot\nabla (\sigma(x)[w_k]-\sigma(y)[w_k])\|_{E_0}\\
        &\le \lambda_k \Vert   x-y\Vert_{E_0}\sup_{z\in B_{\lVert x \rVert \vee \lVert y \rVert;E_0}}\Vert  D\sigma(z)\Vert^2_{L(E_0,L(E_0,E_0))}\\
        &\qquad + \lambda_k \Vert   x-y\Vert_{E_0}\sup_{z\in B_{\lVert x \rVert \vee \lVert y \rVert;E_0}}\Vert  \sigma(z)\Vert_{L(E_0,L(E_0,E_0))}\cdot \\
        &\qquad \qquad \qquad \cdot \sup_{z\in B_{\lVert x \rVert \vee \lVert y \rVert;E_0}}\Vert  D^2\sigma(z)\Vert_{L(E_0,L(E_0,L(E_0,E_0)))}.
        \end{split}
    \end{align}
    Therefore, there exists a non-increasing function $h:\mathbb{R}\to \mathbb{R}_+$ such that
    \begin{align*}
                &\sum_{k\ge 1} \Vert  \sigma_k(x)\Vert^2_{E_0}
                \le \Tr Q \,\|\sigma(x)\|^2 _{L(E_0,E_0)} \le h (\Vert x\Vert_{E_0} ),\\
                & \sum_{k\ge 1} \Vert   \sigma_k(x)-\sigma_k(y)\Vert^2_{E_0}
                \le \Tr Q \, \Vert   x-y\Vert^2_{E_0}\sup_{z\in B_{\lVert x \rVert \vee \lVert y \rVert;E_0}}\Vert  D\sigma(z)\Vert^2_{L(E_0,L(E_0,E_0))}\\
                &\qquad\qquad\qquad\qquad\qquad\leq h (\Vert x\Vert_{E_0}\vee\Vert y\Vert_{E_0} )\Vert x-y\Vert^2_{E_0},\\
                &\sum_k\sup_{x\in B_{R;E_0}}\|\left(\sigma\cdot\nabla \sigma\right)_k(x)\|_{E_0} \le \Tr Q \sup_{x\in B_{R;E_0}} \left(\|D\sigma(x)\|_{L(E_0,L(E_0,E_0))}\|\sigma(x)\|_{L(E_0,E_0)}\right) \le h(R),\\
                & \Big\Vert \left( \sigma\cdot \nabla \sigma\right)(x)  \Big\Vert_{E_0} 
                \le \sum_{k\ge 1} \Vert \left(\sigma\cdot \nabla \sigma\right)_k(x)  \Vert_{E_0} 
                \le \Tr Q \, \|\sigma(x)\|_{L(E_0,E_0)} \|D\sigma(x)\|_{L(E_0,L(E_0,E_0))}\\
                &\qquad\qquad\qquad\qquad\qquad\;\, \le h (\Vert x\Vert_{E_0}),\\
                & \Big\Vert  \left( \sigma\cdot \nabla \sigma\right)(x) - \left(\sigma\cdot \nabla \sigma\right)(y) \Big\Vert_{E_0} 
                \le
                \sum_{k\ge 1} \Vert \left( \sigma\cdot \nabla \sigma\right)_k(x) - \left(\sigma\cdot \nabla \sigma\right)_k(y) \Vert_{E_0} \\
                &\qquad\qquad\qquad\qquad\qquad \qquad\qquad\qquad\qquad\;\;\,\le h (\Vert x\Vert_{E_0}\vee\Vert y\Vert_{E_0} )\Vert x-y\Vert_{E_0}.
    \end{align*}
    Concerning the locally Lipschitzness property for $\sigma^n$, taking $x,y \in E^n$ with $\lVert x \rVert_{E_0}, \lVert y \rVert_{E_0} \le R$,
    \begin{align*}
        \lVert \sigma^n(x)-\sigma^n(y)\rVert_{L(\R^n,E^n)}
        &=\sup_{k\in \{1,\dots,n\}} \lVert \Pi_n (\sigma_k(x)-\sigma_k(y)) \rVert_{E^n} 
        \le \sup_{k\in \{1,\dots,n\}}\sqrt{\lambda_k} \lVert (\sigma(x)-\sigma(y))[w_k] \rVert_{E_{-1}}\\
        &\lesssim \sup_{k\in \{1,\dots,n\}}\sqrt{\lambda_k}  \lVert (\sigma(x)-\sigma(y))[w_k] \rVert_{E_{0}} 
        \le \lVert Q \rVert^\frac12 \,\lVert \sigma(x)-\sigma(y) \rVert_{L(E_{0},E_{0})}\\
        &\le \lVert Q \rVert^\frac12 \, \Vert   x-y\Vert_{E_0}\sup_{z\in B_{\lVert x \rVert \vee \lVert y \rVert;E_0}}\Vert  D\sigma(z)\Vert_{L(E_0,L(E_0,E_0))}\\
        &\le  \lVert Q \rVert^\frac12 \, \Vert   x-y\Vert_{E_0} h(R)\\
        &\le C_R \Vert   x-y\Vert_{E^n}
    \end{align*}
    where we exploited the equivalence of norms in finite dimensional vector spaces.
    Finally, concerning the locally Lipschitzness property for $\left(\sigma \cdot\nabla \sigma \right)^n$ is sufficient to show the same property for $\left(\sigma \cdot\nabla \sigma\right)^n_k$ for every $k$. We have, making use of \eqref{eq:estimate_sigma_nabla_sigma},
    \begin{align*}
        \lVert \Pi_n ((\sigma \cdot\nabla \sigma )_k(x) - (\sigma \cdot\nabla \sigma )_k(y)) \rVert_{E^n} 
        &\le \lVert (\sigma \cdot\nabla \sigma )_k(x) - (\sigma \cdot\nabla \sigma )_k(y) \rVert_{E_{-1}}\\
        &\lesssim \lVert (\sigma \cdot\nabla \sigma )_k(x) - (\sigma \cdot\nabla \sigma )_k(y) \rVert_{E_{0}}\\
        &\le  \lVert Q \rVert \, \Vert   x-y\Vert_{E_0} h(R)\\
        &\le C_R \Vert   x-y\Vert_{E^n}.
    \end{align*}
The proof is complete.
\end{proof}

\subsection{Proof of Lemma \ref{lemma:coherence}} \label{subsection:proof_lemma_coherence}
\begin{proof}
Recall that $\sigma$ takes the form given in \eqref{eq:sigma} only for $x\in B_{R;E_0}$ for some $R>0$. In this region, $\sigma$ is $C^\infty$ and its norm and the norms of its derivatives only depend on the norm of $x$.\\
We have to show Assumption \ref{Assump:sigma}-(b). First, notice that, recalling \eqref{eq:alpha_beta_gamma_delta}, for $x\in E^n$
\begin{align*}
    \sum_{k=1}^n \langle x, \left(\sigma\cdot \nabla \sigma\right)^n_k(x)\rangle_{E_1} 
    &= \sum_{k=1}^n \lambda_k c_N^2(\eta - \alpha - \eta\alpha)\lVert x \rVert^{2\eta-2}_{E_0}\langle x,w_k\rangle_{E_0} \langle x, \Pi_n w_k\rangle_{E_1}\\
    &\qquad + \sum_{k=1}^n \lambda_k c_N^2\lVert x \rVert^{2\eta-2}_{E_0}\Bigl(\lVert x \rVert^{-2}_{E_0} \langle x,w_k\rangle^2_{E_0}(\alpha^2\eta-\alpha(\eta-2)) -\alpha\lVert w_k \rVert^2_{E_0}\Bigr) \lVert x \rVert^2_{E_1}. 
\end{align*}
Exploiting Cauchy-Schwarz inequality and Assumptions \ref{Assump:WienerProcess}, \ref{Assump:Spaces}, \ref{Assump:Projection},
\begin{align*}
   &\sum_{k=1}^n \lambda_k c_N^2(\eta - \alpha - \eta\alpha)\lVert x \rVert^{2\eta-2}_{E_0}\langle x,w_k\rangle_{E_0} \langle x, \Pi_n w_k\rangle_{E_1}  \\ 
    &\qquad\le C_{\alpha,\eta}  \lVert x \rVert^{2\eta-2}_{E_0} \lVert x \rVert^2_{E_1}\sum_{k=1}^n \lambda_k \lVert w_k\rVert^2_{E_1}  \\
    &\qquad\le  C_{\alpha,\eta}  \lVert x \rVert^{2\eta-2}_{E_0} \lVert x \rVert^2_{E_1}.
\end{align*}
Similarly,
\begin{align*}
    &\sum_{k=1}^n \lambda_k c_N^2\lVert x \rVert^{2\eta-2}_{E_0}\Bigl(\lVert x \rVert^{-2}_{E_0} \langle x,w_k\rangle^2_{E_0}(\alpha^2\eta-\alpha(\eta-2)) -\alpha\lVert w_k \rVert^2_{E_0}\Bigr) \lVert x \rVert^2_{E_1}  \\
    & \qquad\le C_{\alpha,\eta}  \lVert x \rVert^{2\eta-2}_{E_0} \lVert x \rVert^2_{E_1} \sum_{k=1}^n \lambda_k   \\
    & \qquad\le C_{\alpha,\eta} \Tr Q\,  \lVert x \rVert^{2\eta-2}_{E_0} \lVert x \rVert^2_{E_1}.
\end{align*}
On the other hand, the second part of Assumption \ref{Assump:sigma}-(b) can be shown in the following way
\begin{align*}
\begin{split}
    &\sum_{k=1}^n \lambda_k c_N^2\rVert x \rVert^{2\eta}_{E_0} \left\Vert \Pi_n w_k - \alpha \frac{1}{\lVert x \rVert_{E_0}^2}x \langle x, w_k\rangle_{E_0} \right\Vert_{E_1}^2  \\
    & \le C  \lVert x \rVert^{2\eta-2}_{E_0} \lVert x \rVert^2_{E_1} \sum_{k=1}^n \lambda_k \lVert w_k \rVert^2_{E_1} +  C  \lVert x \rVert^{2\eta-2}_{E_0} \lVert x \rVert^2_{E_1} \sum_{k=1}^n \lambda_k \\
    & \le C  \lVert x \rVert^{2\eta-2}_{E_0} \lVert x \rVert^2_{E_1} +  C \Tr Q \lVert x \rVert^{2\eta-2}_{E_0} \lVert x \rVert^2_{E_1} \\
    & \le C \lVert x \rVert^{2\eta-2}_{E_0} \lVert x \rVert^2_{E_1}.
\end{split}
\end{align*}
Now, we show that $\sigma$ as defined in \eqref{eq:sigma_full_space} is a suitable example of $\sigma$ defined in the whole $E_0$. For notational convenience, we relabel $\sigma$ with $\hat \sigma$. In this way, we have $\hat \sigma(x) = \psi(\lVert x \rVert_{E_0}) \sigma(x)$, for $\sigma(x)$ which is coherent with all the previous computations. The only non trivial part to show is that $\forall x \in E^n$ and $\forall n\in \N$
\begin{align*}
    \langle x, \left(\hat\sigma\cdot \nabla \hat\sigma\right)^n(x)\rangle_{E_1} \leq
f (\Vert x\Vert_{E_0} ) \lVert x \rVert_{E_1}^2,
\end{align*}
due to the presence of the gradient. In particular, by Leibniz rule
\begin{align*}
    \left(\hat\sigma\cdot \nabla \hat\sigma\right)^n(x) = \psi^2(\Vert x\Vert_{E_0})\left(\sigma\cdot \nabla \sigma\right)^n(x) + \psi(\Vert x\Vert_{E_0}) \frac{\psi'(\Vert x\Vert_{E_0})}{\Vert x\Vert_{E_0}} \sum_{k=1}^n  \langle  x, \sigma_k(x) \rangle_{E_0} \Pi_n \sigma_k(x),
\end{align*}
so that
\begin{align}\label{eq:computation_for_sgs}
\begin{split}
    \langle x, \left(\hat\sigma\cdot \nabla \hat\sigma\right)^n(x)\rangle_{E_1}
    &= \psi^2(\Vert x\Vert_{E_0})\langle x, \left(\sigma\cdot \nabla \sigma\right)^n(x) \rangle_{E_1} \\
    &\quad+ \psi(\Vert x\Vert_{E_0}) \frac{\psi'(\Vert x\Vert_{E_0})}{\Vert x\Vert_{E_0}} \sum_{k=1}^n  \langle  x, \sigma_k(x) \rangle_{E_0} \langle x, \Pi_n \sigma_k(x) \rangle_{E_1}.
\end{split}
\end{align}
Recalling that $\psi$ vanishes close to the origin, the bound for the first addendum in the right hand side of \eqref{eq:computation_for_sgs} can be bound exactly as in the previous computations. We focus on the second addendum. By Cauchy-Schwarz inequality
\begin{align*}
    \psi(\Vert x\Vert_{E_0}) &\frac{\psi'(\Vert x\Vert_{E_0})}{\Vert x\Vert_{E_0}} \sum_{k=1}^n  \langle  x, \sigma_k(x) \rangle_{E_0} \langle x, \Pi_n \sigma_k(x) \rangle_{E_1}\\
    &\le  \psi(\Vert x\Vert_{E_0}) \psi'(\Vert x\Vert_{E_0}) \lVert \sigma(x) \rVert_{L(E_0,E_0)}  \rVert x \rVert_{E_1} \sum_{k=1}^n \sqrt{\lambda_k}  \rVert \sigma^n_k (x)\rVert_{E_1}\\
    &\le  \psi(\Vert x\Vert_{E_0}) \psi'(\Vert x\Vert_{E_0}) \lVert \sigma(x) \rVert_{L(E_0,E_0)}  \rVert x \rVert_{E_1}  \left( \Tr Q\right)^\frac12 \left( \sum_{k=1}^n   \rVert \sigma^n_k (x)\rVert^2_{E_1}\right)^\frac12,
\end{align*}
therefore, recalling that both $\psi$ and $\psi'$ vanish close to the origin, we conclude the proof exploiting the bound already obtained in the previous computation for the term $\sum_{k=1}^n   \rVert \sigma^n_k (x)\rVert^2_{E_1}$.
\end{proof}

\printbibliography

\end{document}